\documentclass[11pt]{article}

\usepackage{color}
\usepackage{graphicx}
\usepackage{amsmath}
\usepackage{cases}
\usepackage{amsthm}
\usepackage{amssymb}
\usepackage{bm}
\usepackage{cases}
\usepackage{amsfonts}
\usepackage{latexsym,enumerate,amscd,verbatim}
\usepackage{mathrsfs}
\usepackage[pdftex,colorlinks=true,bookmarks=false,citecolor=blue]{hyperref}
\usepackage{dsfont}
\usepackage{setspace}

\usepackage[figuresright]{rotating}%
\usepackage[title]{appendix}%
\usepackage{xcolor}%
\usepackage{textcomp}%
\usepackage{manyfoot}%
\usepackage{booktabs}%
\usepackage[lined]{algorithm}%
\usepackage{algorithmicx}%
\usepackage{algpseudocode}%
\usepackage{listings}%

\allowdisplaybreaks
\numberwithin{equation}{section}

\newtheorem{theorem}{Theorem}[section]
\newtheorem{lemma}{Lemma}[section]
\newtheorem{proposition}{Proposition}[section]

\newtheorem{definition}{Definition}[section]

\newtheorem{example}{Example}[section]
\newtheorem{remark}{Remark}[section]

\makeatletter

\renewcommand{\fnum@algorithm}{}

\newcommand{\vi}[2]{\ensuremath{\textup{VI}(#1,#2)}}

\newcommand{\inner}[2]{\left\langle#1,#2\right\rangle}
\newcommand{\norm}[1]{\|#1\|}

\def\O{\Omega}
\def\o{\omega}
\def\dbE{\mathds{E}}
\def\mrn{R^n}
\def\mn{\mathbb{N}}

\renewcommand{\@seccntformat}[1]{\csname the#1\endcsname.\hspace{0.5em}}
\makeatother

\title{A prediction-correction ADMM  for
multistage \\ stochastic variational inequalities\footnote{This work was supported by the NSF of China under grants 12071324 and 11931011.}}

\author{Ze You\thanks{School of Mathematical  Sciences, Sichuan Normal University, Chengdu  610066, China. Email: \texttt{ZeYou9205@163.com}.
 }\;\; \;\;
Haisen Zhang\thanks{Corresponding author. School of Mathematical Sciences, Sichuan Normal University, Chengdu 610066,  China. Email: \texttt{haisenzhang@yeah.net}.}\;\; \;\;
}
\date{}

\begin{document}
\maketitle

\begin{spacing}{1.0}

\begin{abstract}
 The multistage stochastic variational inequality is reformulated into a variational inequality with separable structure through introducing a new variable. The prediction-correction ADMM  which was originally proposed in [B.-S. He, L.-Z. Liao and M.-J. Qian, J. Comput. Math., 24 (2006), 693--710]  for solving deterministic variational inequalities in finite dimensional spaces  is adapted to solve the multistage stochastic variational inequality. Weak convergence of the  sequence generated by that algorithm is proved under the conditions of monotonicity  and Lipschitz continuity. When the sample space is a finite set, the corresponding multistage stochastic variational inequality is actually defined on a finite dimensional Hilbert space and the strong convergence of the algorithm naturally holds true. Some numerical examples are given to show the efficiency of the algorithm.
\end{abstract}

\vspace{+0.3em}


\vspace{+0.3em}


\section{Introduction}

Variational inequalities (VIs), as the first order necessary conditions of convex programming problems, equilibrium problems and optimal control problems, provide us with a powerful tool to solve various types of application problems. The deterministic variational inequalities have been extensively studied from the aspects of both the theories and algorithms in the past decades, for more details, see \cite{FiniDim03I} and \cite{FiniDim03II}.

In many practical problems, especially  in  finance, economy and management, the decision makers have to face the uncertainty brought by some stochastic factors. When the impact of the stochastic factors cannot be ignored, the deterministic variational inequalities may not be the suitable/precise models for those problems. In order to model the problems with uncertainty, in recent years, various types of stochastic variational inequalities (SVIs) are proposed.

The first class of stochastic variational inequalities which has attracted the attention of scholars is the one-stage stochastic variational inequalities. Based on whether the stochastic information is known or not before the decision-making, three kinds of one-stage stochastic variational inequalities are proposed, i.e., the Wait-and-See model, the Expected-Value model \cite{Robinson1999} and the Expected-Residual-Minimization model \cite{Chen Fukushima 2005}.

In a Wait-and-See model, the stochastic information is assumed to be known when one makes a decision, and the solution to such stochastic variational inequality is a response function of scenarios or a function of a certain random variable. Solving the Wait-and-See stochastic variational inequality is equivalent to solving individually a collection of variational inequalities with stochastic parameters. The Expected-Value stochastic variational inequality is a deterministic variational inequality with the map represented by the expected value of some map with stochastic parameters. In the Expected-Value model, the decision should be made before the stochastic information is observed and hence the decision set is independent of the stochastic factors. The basic idea of Expected-Residual-Minimization model comes from finding a common (deterministic) solution to a collection of variational inequalities with stochastic parameters. In general, such common solution might not exist. The Expected-Residual-Minimization model focuses on finding a (deterministic) solution through minimizing the expectation of a residual function for  parameterized variational inequalities. Particularly, when the optimal value of that minimization problem is zero, any minimizer is almost surely a common solution of the corresponding parameterized  variational inequalities. Both the Expected-Value model and the Expected-Residual-Minimization model are called the Here-and-Now model.

The one-stage model does not take into account the increasing levels of observed information in the process of decision-making. In practice, there exist a large number of problems in which the decisions can be made step by step. Clearly, in such situation, the observed information until the current step may affect the decision of next step. In other words, the decision in each step should be a response to the historical observation data. In order to model the dynamical decision problems that the decision makers can use the historical observation data at each step when he/she makes a decision, the two-stage and multistage stochastic variational inequalities are proposed.

Under the assumption that the uncertainty is described by a random vector and the random vector can be observed completely in the second stage, the two-stage stochastic variational inequality is to find a pair of solutions to a coupled stochastic variational inequality system: A Here-and-Now solution to an Expected-Value model in stage one and a solution to the Wait-and-See model in stage two. We refer the readers to \cite{Chen pong wets2017} and \cite{sun chen JORSC2021} for the definitions and examples of the two-stage stochastic variational inequalities.
In \cite{Rockafellar Wets17}, the authors first introduced the notion of multistage stochastic variational inequality with nonanticipativity constraints. It is showed that, both the one-stage Expected-Value type stochastic variational inequality and the two-stage stochastic variational inequality are contained in the framework of the multistage model of \cite{Rockafellar Wets17}.

There are two main approaches to solve the two-stage and multistage stochastic variational inequalities. The first one is the sample average approximation (SAA) method. The SAA method was first proposed for solving stochastic programming problems (see, e.g., \cite{Shapiro handbook SP 2009}), one-stage stochastic variational inequalities of Expected-Value model (see, e.g., \cite{Robinson1999}) and the one-stage stochastic variational inequalities of Expected-Residual-Minimization model (see, e.g., \cite{Chen Fukushima 2005,ChenWetsZhang SAA 2012}). Then it was extended to solve the two-stage stochastic variational inequalities and two-stage stochastic generalized equations, see \cite{Chen pong wets2017,Chen sha sun 2019siopt}.  The second one is the progressive hedging algorithm (PHA). The PHA was first proposed in \cite{Rockafellar wets 1991}  for solving the stochastic programming problem with nonanticipativity constraint and then adapted by Rockafellar and Sun in \cite{Rockafellar Sun 2019} to solve the two-stage and multistage stochastic variational inequalities in the discrete cases.  Recently,  the PHA was developed to solve various types of two-stage and multistage stochastic variational inequalities and stochastic games, see, for instance, \cite{RockafellarSun 2020,SunYangYaoZhang 2020,ZhangSunXu game pha 2019}. When the probability space under consideration is not discrete, \cite{Chen sun xu 2019} proposed a discrete approximation method for two-stage stochastic linear complementarity problems.

In this paper, we shall adapt the definition of the multistage stochastic variational inequalities given in \cite{Rockafellar Wets17}. A slight difference is that, rather than  the discrete cases, we shall give the definition of multistage stochastic variational inequalities in a general probability space. One main motivation of extending the concept of multistage stochastic variational inequalities into the general probability space is for its particular use in stochastic optimal control problems. We refer the readers to  Example \ref{ex SOCP} for more details. We shall see that the multistage stochastic variational inequality is indeed a variational inequality defined on the Hilbert space of square integrable random vectors. Theoretically, it can be solved by the projection-type algorithms for deterministic variational inequalities. The main difference between the deterministic variational inequality and the multistage stochastic variational inequality is that, in the stochastic case, the calculation of the projection onto a subset of random vectors is much more complicated. Especially, to find the projection of a random vector onto the nonanticipativity subspace one needs to compute a collection of conditional expectations, which is quite different from the usual metric projection onto a nonempty closed convex set (see Remark \ref{Remark3.1} for more details).

One of the key idea of PHA is treating  the metric projection onto a nonempty closed convex subset and the projection onto the nonanticipativity subspace separately. It was showed in \cite{sun xu zhang 2020} and \cite{mu yang 2020} that, the PHA is equivalent to the  alternating direction method of multipliers (ADMM). The ADMM, as an extension of inexact augmented Lagrangian method  (ALM),  was first proposed by Glowinski and Marrocco in \cite{Glowinski75}. In the past few decades, the ADMM and its various extensions have been deeply studied by many scholars for solving mathematical programming problems and variational inequality problems under the deterministic framework. For related work, we refer  the reader to the review articles \cite{Boyd2011} and \cite{Hebingsheng2018}.

Note that the PHA (or the ADMM) for multistage stochastic variational inequalities is an implicit iterative algorithm. In each iteration one needs to solve a collection of variational inequalities with stochastic parameters. In this paper,  we shall adapt the prediction-correction ADMM, which originally proposed in \cite{HeJCM2006}  for solving deterministic finite dimensional  variational inequalities with separable structures, to solve the multistage stochastic variational inequalities. Different from the PHA, the prediction-correction ADMM  is an explicit iterative algorithm, in which  the calculation of each iteration becomes relatively easier than that of PHA. The weak convergence of the prediction-correction ADMM for multistage stochastic variational inequalities is proved in the general probability space under the conditions of monotonicity and Lipschitz continuity. When the sample space is a finite set, as discussed in \cite{Rockafellar Wets17}, the  multistage stochastic variational inequalities is actually  defined on a finite dimensional Hilbert space and the strong convergence of the algorithm  naturally  holds true. Two numerical examples are given in such case to show the efficiency of the algorithm.

The rest of this paper is organized as follows. Some basic notions and results in probability, set-valued and variational analysis  and multistage stochastic variational inequalities are given in Section 2. In Section 3, the prediction-correction ADMM for multistage stochastic variational inequalities and its weak convergence are studied.  The discrete cases and numerical examples are discussed  in Section 4. Some concluding remarks are given in Section 5. A proof of a technical result is given in the Appendix.

\section{Preliminaries}

In this section, we recall some basic notions and preliminary results in probability and set-valued and variational analysis. Then we give the definition of the multistage stochastic variational inequalities in the general probability space.

\subsection{Concepts and results in probability}

First, we recall some basic concepts and results in probability. We refer the readers to  \cite{A N Shiryaev Probability 96} for more details.

Let $(\Omega,\mathscr{F},P)$ be a complete probability space. Here $\Omega$ is the sample space, $\mathscr{F}$ is a $\sigma$-field defined on $\Omega$,  and $P$ is a  probability measure defined on $(\Omega,\mathscr{F})$. Any element of $\Omega$, denoted by $\omega$, is called a sample point.  As usual, when the context is clear, we omit the $\omega$ ($\in \Omega$) argument in the defined maps/functions. Denote by $\mathscr{O}$ the collection of all $P$-null sets. We say that a property holds almost surely (a.s.) if there
is a set $A\in \mathscr{O}$ such that the property holds for every $\o\in \O\setminus A$.  Denote by $\mathscr{F}_{0}$ the trivial $\sigma$-field which only contains $\emptyset$ and  $\O$.  Let $\mathscr{A}$  be a family of subsets of $\Omega$. The smallest $\sigma$-field containing $\mathscr{A}$, denoted by $\sigma(\mathscr{A})$, is called the  $\sigma$-field generated by $\mathscr{A}$. For a metric space $Y$, denote by $\mathscr{B}(Y)$ the Borel  $\sigma$-field of $Y$. Let $R^n$ ($n\in \mathbb{N}$) be an $n$-dimensional Euclidean space with Borel $\sigma$-field $\mathscr{B}(R^n)$.  A map $\xi: \Omega\rightarrow R^n$ is called an  $\mathscr{F}$-measurable random vector if
\begin{equation*}
  \xi^{-1}(A):=\big\{\omega\in\Omega\ \big|\ \xi(\omega)\in A\}\in\mathscr{F},\quad \forall\ A\in\mathscr{B}(R^n).
\end{equation*}
Denote by $\mathcal{L}^0(\Omega, \mathscr{F}, R^n)$ the set of all $\mathscr{F}$-measurable random vectors.  When a random vector $\xi\in \mathcal{L}^0(\Omega, \mathscr{F}, R^n)$  is  integrable with respect to the probability measure $P$, we call the integral
$\dbE\ \xi =\int_{\Omega}\xi(\o)P(d\o) $
the expectation of $\xi$.  Let $\inner{ \cdot}{\cdot}$ and $\mid\cdot \mid$ be respectively the inner product and norm in $R^n$. Denote by  $\mathcal{L}^2(\Omega,\mathscr{F}, R^n)$ the Hilbert  space of all  the square-integrable random vectors taking values in $\mrn$, i.e.,
\begin{equation}\label{L2definition}
\mathcal{L}^2(\Omega,\mathscr{F}, R^n)=\Big\{\xi\in \mathcal{L}^0(\Omega, \mathscr{F}, R^n) \ \big|\ \dbE| \xi|^2 <+\infty\Big\}.
\end{equation}
For any $\xi,\eta\in \mathcal{L}^2(\Omega,\mathscr{F}, R^n)$, the norm of $\xi$ is defined by
$$\|\xi\|_{\mathcal{L}^2}:=\Big[\dbE| \xi|^2\Big]^{\frac{1}{2}}=\Bigg[\int_{\O}| \xi(\o)|^2P(d\o)\Bigg]^{\frac{1}{2}},$$
and  the  inner product of $\xi$ and $\eta$ is defined by
\begin{equation*}
\inner{\xi}{\eta}_{\mathcal{L}^2}:=\dbE \inner{\xi}{\eta} =\int_{\O}\inner{\xi(\o)}{\eta(\o)} P(d\o).
\end{equation*}
Let $\{x^k\}_{k=1}^{\infty}$ be a sequence in $\mathcal{L}^2(\Omega,\mathscr{F}, R^n)$. We write $x^k\rightharpoonup  x$ to indicate that the sequence $\{x^k\}_{k=1}^{\infty}$ converges  weakly to $x$   and $x^k \rightarrow x$  to indicate that the sequence $\{x^k\}_{k=1}^{\infty}$ converges  strongly   to $x$.

\begin{definition}\cite[Definition 1, P. 213]{A N Shiryaev Probability 96}
Let $\xi\in\mathcal{L}^0(\Omega, \mathscr{F}, R^n)$, $\dbE\; |\xi|<+\infty$ and $\mathscr{G}$ be a sub-$\sigma$-field of $\mathscr{F}$. The conditional expectation of $\xi$ with respect to the $\sigma$-field $\mathscr{G}$, denoted by $\dbE\big[\xi|\mathscr{G}\big]$, is a random vector such that
\begin{enumerate}[{\rm i)}]
  \item $\dbE\big[\xi|\mathscr{G}\big]:\O\to \mrn$ is $\mathscr{G}$-measurable;
  \item   $\int_{A}\xi(\o)~P(d\o)=\int_{A}\dbE\big[\xi|\mathscr{G}\big](\o)~P(d\o),\ \forall\ A\in\mathscr{G}. $
\end{enumerate}
\end{definition}

The conditional expectation has the following basic properties.

\begin{lemma}\cite[P.215]{A N Shiryaev Probability 96}\label{properties for CE}
Suppose that $\xi,\eta\in \mathcal{L}^0(\Omega, \mathscr{F}, R^n)$, $\dbE |\xi|<+\infty$, $\dbE |\eta|<+\infty$, $\dbE \ \big| \inner{\xi}{\eta} \big| <+\infty$,  $\mathscr{F}_0=\{\emptyset, \O\}$ and $\mathscr{G}$ is a sub-$\sigma$-field of $\mathscr{F}$. Then, the following asserts hold true:
\begin{enumerate}[{\rm i)}]
  \item If $v$ is a constant vector and $\xi=v$ a.s., then $\dbE\big[\xi|\mathscr{G}\big]=v$ \mbox{a.s.};
  \item If $\xi$ is $\mathscr{G}$-measurable, then $\dbE\big[\xi|\mathscr{G}\big]=\xi$ a.s.;
  \item $\dbE\big[\xi|\mathscr{F}_{0}\big]=\dbE ~\xi$ a.s.;
  \item $\dbE\big\{\dbE\big[\xi|\mathscr{G}\big]\big\}=\dbE~ \xi$;
  \item  If $\xi$ is $\mathscr{G}$-measurable, then $
\dbE\big[\inner{\xi}{\eta}|\mathscr{G}\big]=\inner{\xi}{ \dbE\big[\eta|\mathscr{G}\big]}$ a.s.
\end{enumerate}
\end{lemma}

\begin{remark}\label{Rem projection CE}
Let $\mathscr{G}$ be a sub-$\sigma$-field included in $\mathscr{F}$. Define
$$ \mathcal{L}^2(\Omega,\mathscr{G}, R^n)=\Big\{\xi\in \mathcal{L}^2(\Omega,\mathscr{F}, R^n) \ \big|\ \xi \mbox{ is } \mathscr{G}-\mbox{measurable} \Big\},$$
and let $\xi\in \mathcal{L}^2(\Omega,\mathscr{F}, R^n)$.
Clearly, $ \mathcal{L}^2(\Omega,\mathscr{G}, R^n)$ is a closed linear subspace of  $\mathcal{L}^2(\Omega,\mathscr{F}, R^n)$ and, by Lemma \ref{properties for CE},
$$\dbE\inner{\xi-\dbE\big[\xi|\mathscr{G}\big]}{v}
=\dbE\inner{\xi}{v}-\dbE\big\{\dbE\big[\inner{\xi}{v}|\mathscr{G}\big]\big\}
=0,\quad \forall\;v \in \mathcal{L}^2(\Omega,\mathscr{G}, R^n).
$$
Then, for any $\eta\in \mathcal{L}^2(\Omega,\mathscr{G}, R^n)$, by  $\eta-\dbE [\xi|\mathscr{G}]\in  \mathcal{L}^2(\Omega,\mathscr{G}, R^n)$ we have
\begin{eqnarray*}
&&\dbE |\xi-\dbE[\xi|\mathscr{G}]|^2\\
&=&\dbE |\xi-\eta|^2+\dbE |\eta-\dbE [\xi|\mathscr{G}]|^2
+2 \dbE \inner{\xi-\eta}{\eta-\dbE[\xi|\mathscr{G}]}\\
&=&\dbE |\xi-\eta|^2-\dbE |\eta-\dbE [\xi|\mathscr{G}]|^2
+2 \dbE \inner{\xi-\dbE [\xi|\mathscr{G}]}{\eta-\dbE [\xi|\mathscr{G}]}\\
&=&\dbE |\xi-\eta|^2-\dbE |\eta-\dbE [\xi|\mathscr{G}]|^2\\
&\le&\dbE |\xi-\eta|^2.
\end{eqnarray*}
Therefore, $\dbE~[\xi|\mathscr{G}]$ is the metric projection  of $\xi$ onto the subspace $\mathcal{L}^2(\Omega,\mathscr{G}, R^n)$ in the sense of the $\mathcal{L}^2$ norm.
\end{remark}

\begin{example}\cite[P.78]{A N Shiryaev Probability 96}\label{example for CE}
Assume that the $\sigma$-field $\mathscr{G}$ is generated by disjoint  subsets $A_1, A_2,...,A_m$ ($m\in \mn$) with $A_i\in  \mathscr{F}$, $P(A_i)>0$, $i=1,2,...,m$ and $\O=\cup_{i=1}^{m}A_i$.  Then, for any $\xi\in \mathcal{L}^2(\Omega,\mathscr{F}, R^n)$,
$$\dbE~[\xi|\mathscr{G}]=\sum_{i=1}^{m}\frac{\dbE~ [\xi\chi_{A_i}]}{P(A_i)}\chi_{A_i},$$
where $\chi_{A_i}$ is the characteristic function of $A_i$, i.e.,
$$\chi_{A_i}(\o)=\left\{
\begin{array}{l}
1, \qquad \o\in A_i,\\
0, \qquad  \o\notin A_i,\\
\end{array}\right.\quad \forall\ i=1,2,...,m.$$
Let $B_1, B_2,...,B_\ell$ ($\ell\in \mn$) be disjoint  subsets, $u_{1},u_{2},...,u_{\ell}\in R^n$ and define $\xi=\sum_{j=1}^{\ell} u_{j} \chi_{B_j}$.  Then we have
$$\dbE~[\xi|\mathscr{G}]=\sum_{i=1}^{m}\sum_{j=1}^{\ell}u_{j} P(B_j| A_{i})\chi_{A_i},$$
where $P(B_j| A_{i})=P(A_{i}\cap B_j)/P(A_i)$ is the probability of $B_j$ under the condition that $A_i$ has occurred.
\end{example}

\subsection{Set-valued and variational analysis}

In this subsection, we introduce some elemental results in set-valued  and variational analysis. We refer the readers to \cite{Aubin90} for more details.

Let $(X, \mathscr{S}, \mu)$ be a complete $\sigma$-finite measure space, $Y$ be a complete separable metric space. A set valued map $\Phi:X\rightsquigarrow Y$ is characterized by its graph $Gph(\Phi)$, a subset of the product space $X\times Y$ defined by
$$Gph(\Phi):=\Big\{(x,y)\in X\times Y\ \Big|\ y\in \Phi(x)\Big\}.$$
For any $x\in X$, the $\Phi(x)$, which is a subset of $Y$,  is called the value of $\Phi$ at $x$. The
domain of $\Phi$ is the subset of elements $x\in X$ such that $\Phi(x)$ is
nonempty, i.e., $Dom(\Phi):=\{x\in X\ |\ \Phi(x)\neq \emptyset\}.$
The image of $\Phi$ is defined by $Im(\Phi):=\cup_{x\in X}\Phi(x).$
Suppose that $\Phi:X\rightsquigarrow Y$ is a set-valued map with closed values. $\Phi$ is called $\mathscr{S}$-measurable if
$$\Phi^{-1}(A):=\{x\in X\ |\ \Phi(x)\cap A\neq \emptyset\}\in \mathscr{S},\quad \forall \ A\in \mathscr{B}(Y).$$
A single-valued map $\varphi:X\to Y$ is called a measurable selection of $\Phi$ if $\varphi$ is $\mathscr{S}$-measurable and $\varphi(x)\in \Phi(x)$ for
$\mu$-a.e. $x\in X$.
\begin{lemma}\cite[Theorem 8.1.3, page 308]{Aubin90} \label{measurable-selec}
Let $(X, \mathscr{S}, \mu)$ be a complete $\sigma$-finite measurable space, $Y$ a complete separable metric space, $\Phi:X\rightsquigarrow Y$ a measurable set-valued map with nonempty closed values. Then there exists a measurable selection of $\Phi$.
\end{lemma}

Next, we introduce the concepts of monotonicity for a set-valued map.  Let $\mathcal{H}$ be a Hilbert space with norm $\|\cdot\|_{\mathcal{H}}$ and inner product $\inner{\cdot}{\cdot}_{\mathcal{H}}$.

\begin{definition}\label{monotonedefinition}
\cite[Definition 3.5.1,3.5.4]{Aubin90}
A set-valued map $\Phi: \mathcal{H}\rightsquigarrow \mathcal{H}$ is monotone if its graph is monotone in the sense that
$$\langle u-v, x-y\rangle _{\mathcal{H}} \geq 0,\quad \forall \ (x,u), (y,v)\in Gph(\Phi).$$
A set-valued map $\Phi$ is maximal monotone if
\begin{enumerate}[{\rm i)}]
  \item  $\Phi$ is  monotone;
  \item  there is no other monotone set-valued map whose graph strictly contains the graph of $\Phi$.
\end{enumerate}
\end{definition}

\begin{lemma}\label{maximonotoweakstrongly}
Let $\Phi$ be maximal monotone. Then its graph is weakly-strongly closed in the sense that if $x_{n}$ converges weakly to $x$ and if $u_{n}\in \Phi(x_{n})$, $u_{n}$ converges strongly to $u$, then $u\in \Phi(x)$.
\end{lemma}
\begin{proof} The proof is similar to \cite[Proposition 3.5.6]{Aubin90}, so we omit it.
\end{proof}

Let  $\mathcal{S}$ be a nonempty closed convex subset of $\mathcal{H}$. The normal cone $N_{\mathcal{S}}(x)$ of $\mathcal{S}$ on $x$ is defined by
$$N_{\mathcal{S}}(x)=\Big\{ v\in \mathcal{H}\ \Big|\
\inner{v}{y-x}_{\mathcal{H}}\le 0,\quad \forall\ y\in   \mathcal{S}\Big\}.$$
The metric projection of $x\in \mathcal{H}$ onto $\mathcal{S}$ with norm $\|\cdot\|_{\mathcal{H}}$ is defined by
\begin{equation}\label{metric proj}
\Pi_{\mathcal{S}}(x):=\Big\{ \bar x\in \mathcal{S}\ \Big|\
\|x-\bar x\|_{\mathcal{H}}=\inf_{y\in \mathcal{S}}\|x-y\|_{\mathcal{H}} \Big\}.
\end{equation}
By the fundamental theory in convex analysis, $\Pi_{\mathcal{S}}(x)$ is a  singleton and satisfies
\begin{enumerate}[{\rm i)}]
  \item $ \inner{x-\Pi_{\mathcal{S}}(x)}{y-\Pi_{\mathcal{S}}(x)}_{\mathcal{H}}\le 0, \quad \forall \ y\in \mathcal{S};$
  \item  $\|\Pi_{\mathcal{S}}(x)-\Pi_{\mathcal{S}}(y) \|_{\mathcal{H}}\le \|x-y \|_{\mathcal{H}}, \quad  \forall \ x,y\in \mathcal{H}.$
\end{enumerate}

\begin{definition}\label{VIdefinition}
Let $\mathcal{H}$ be a Hilbert space, $\mathcal{S}\subseteq \mathcal{H}$ be a nonempty closed convex set, and $F:\mathcal{H}\rightarrow \mathcal{H}$ be a given map. The variational inequality problem \vi{F}{ \mathcal{S}} is to
find an $x^*\in \mathcal{S}$ such that
$$\langle F(x^*), y-x^*\rangle_{\mathcal{H}} \geq 0 ,\quad \forall\ y\in \mathcal{S}.$$
\end{definition}

Clearly, $x^*\in \mathcal{S}$ is a solution to \vi{F}{ \mathcal{S}} if and only if $0\in F(x^*)+ N_{\mathcal{S}}(x^*).$

\begin{lemma}\cite[Theorem 3]{Rock70maximonoto}\label{maxnormalmaximonotone}
Let $\mathcal{H}$ be a Hilbert space, $\mathcal{S}\subseteq \mathcal{H}$ be a nonempty closed convex set, and, $F:\mathcal{H}\rightarrow \mathcal{H}$ be a  monotone hemi-continuous  map (i.e. continuous from line segments in $\mathcal{H}$ to the weak topology
in $\mathcal{H}$), then the set-valued map $x\rightsquigarrow F(x)+ N_{\mathcal{S}}(x)$ is maximal monotone.
\end{lemma}

\subsection{Multistage stochastic variational
inequalities}

In this subsection, we introduce the definition of the multistage stochastic variational inequalities.  The concept of multistage stochastic variational inequalities was first introduced in \cite{Rockafellar Wets17} in the cases that the uncertainty was represented by some discrete random vectors. We shall see that the definition given in \cite{Rockafellar Wets17} also fits for a much general case.

Let $N\in \mathbb{N}$, $m_{1},m_{2},...,m_{N-1}\in \mathbb{N}$, $\xi_j\in \mathcal{L}^0(\Omega,\mathscr{F}, R^{m_{j}})$, $j=1,2,...,N-1$, $\mathscr{O}$ be the collection of all the $P$-null sets.  For $i\in \mathbb{N}$, $1\le i\le N-1$, define
\begin{equation}\label{Fi}
\mathscr{F}_{i}:=\sigma\Big\{\sigma\big(\xi_{1},\xi_{2},...,\xi_{i}\big)\cup \mathscr{O}\Big\}.
\end{equation}
Here, $\sigma\big(\xi_{1},\xi_{2},...,\xi_{i}\big)$ is the $\sigma$-field  generated by $\xi_1,\xi_2,..., \xi_{i}$ and we call $\mathscr{F}_{i}$ the augmented $\sigma$-field  generated by $\xi_1,\xi_2,..., \xi_{i}$.
Let  $\mathscr{F}_{0}:=\{\emptyset,\Omega\}$. It is clear that
$
\mathscr{F}_{0}\subset\mathscr{F}_{1}\subset\mathscr{F}_{2}\subset\cdots\subset\mathscr{F}_{N-1}\subset \mathscr{F}.$ Let $n_{0},n_{1},...,n_{N-1}\in \mathbb{N}$,  $n_{0}+n_{1}+\cdots+n_{N-1}=n$.
Define
\begin{equation}\label{nonantipaspaceN}
\mathcal{N}:=\Big\{x=(x_{0},x_{1},...,x_{N-1})\in \mathcal{L}^2(\Omega,\mathscr{F}, R^n) \ \big|\ x_{i}\in \mathcal{L}^0(\Omega, \mathscr{F}_{i}, R^{n_i}), i=0,1,...,N-1\Big\},
\end{equation}
i.e., $\mathcal{N}$  is the set of all $x=(x_{0},x_{1},...,x_{N-1})\in \mathcal{L}^2(\Omega,\mathscr{F}, R^n)$ such that for any $i=0,1,...,N-1$, $x_{i}$ is an $\mathscr{F}_{i}$-measurable $R^{n_i}$-valued random vector.  Clearly, $\mathcal{N}$ is a closed linear subspace of $\mathcal{L}^2(\Omega,\mathscr{F}, R^n)$. We  call $\mathcal{N}$ the $nonanticipativity$ subspace.

\begin{remark}
By the definition of  $\mathscr{F}_{i}$, $(\Omega,\mathscr{F}_i, P)$ is a complete probability space for any  $i=1,...,N-1$. In addition, by Doob-Dynkin's Lemma, $x_{0}$ is a deterministic vector in $R^{n_0}$, and, for $i=1,2,...,N-1$,  there are Borel measurable maps $\varphi_{i}$ such that $x_{i}=\varphi_{i}(\xi_{1},\xi_{2},...,\xi_{i})$ \mbox{a.s.}
\end{remark}

Also, we consider the orthogonal complementary subspace of $\mathcal{N}$ defined by
\begin{equation}\label{nonantipaspaceM}
\mathcal{M}=\mathcal{N}^\perp=\Big\{y=(y_{0},y_{1},...,y_{N-1})\in \mathcal{L}^2(\Omega,\mathscr{F}, R^n) \ \Big|\ \dbE\inner{x}{y}=0,\forall x\in\mathcal{N}\Big\}.
\end{equation}
By the properties of conditional expectation, for any $x\in \mathcal{N}$ and  any $y\in \mathcal{L}^2(\Omega,\mathscr{F}, R^n)$,
$$
\dbE\inner{x}{y}=\sum_{i=0}^{N-1} \dbE\inner{x_{i}}{y_{i}}
=\sum_{i=0}^{N-1} \dbE\big[\dbE\inner{x_{i}}{y_{i}}|\mathscr{F}_{i}\big]
=\sum_{i=0}^{N-1} \dbE\inner{x_{i}}{\dbE\big[ y_{i}|\mathscr{F}_{i}\big]}.$$
Then, we have
\begin{equation*}
\mathcal{M}=\Big\{y=(y_{0},y_{1},...,y_{N-1})\in \mathcal{L}^2(\Omega,\mathscr{F}, R^n) \ \Big| \  \dbE \big[y_{i} | \mathscr{F}_{i}\big]=0 \ \mbox{a.s.},\ \forall\ i= 0,1,\cdots,N-1\Big\}.
\end{equation*}

Let $C_{i}:\O\rightsquigarrow R^{n_{i}}$ be given $\mathscr{F_{i}}$-measurable set-valued map  with {\em nonempty closed convex values}, $\ i= 0,1,\cdots,N-1$. Consider the set-valued map $C:\O\rightsquigarrow \mrn$ such that
$$C(\o)=C_{0}(\omega)\times C_{1}(\omega)\times\cdots\times C_{N-1}(\omega)\subset R^{n_{0}}\times R^{n_{1}}\times\cdots\times R^{n_{N-1}}, \quad \mbox{a.s.} \ \o\in \O.$$
Using set-valued map $C$ we define a subset  $\mathcal{C}$ of $\mathcal{L}^2(\Omega,\mathscr{F}, R^n)$ by
\begin{equation}\label{constraintsetC}
\mathcal{C}:=\Big\{x=(x_{0},x_{1},...,x_{N-1})\in \mathcal{L}^2(\Omega,\mathscr{F}, R^n)\ \Big|\ x_{i}(\omega)\in C_{i}(\omega)\ \mbox{a.s.}\ \o\in\O, i= 0,\cdots,N-1 \Big\},
\end{equation}
i.e., each element of $\mathcal{C}$ is an $\mathcal{L}^2$-integrable selection of set-valued map $C$. By the definition of $\mathcal{C}$ and Lemma \ref{measurable-selec}, the set $\mathcal{C}$ is a nonempty closed convex set of  $\mathcal{L}^2(\Omega,\mathscr{F}, R^n)$.

The multistage stochastic variational inequality considered in this paper is defined as follows.

\begin{definition}\label{Mstage SVI def}
Let $(\Omega,\mathscr{F},P)$ be a complete probability space, $F$ be a given map from $\mathcal{L}^2(\Omega,\mathscr{F},R^n)$ to $\mathcal{L}^2(\Omega,\mathscr{F}, R^n)$,  $\mathcal{N}$  and $\mathcal{C}$ be defined by \eqref{nonantipaspaceN}  and \eqref{constraintsetC} respectively, the multistage stochastic variational inequality $\mbox{MSVI}(F,\mathcal{C}\cap\mathcal{N} )$ is: To find $x^*\in\mathcal{C}\cap\mathcal{N}$ such that
\begin{equation}\label{SVIbasicdefinition}
-F(x^*)\in N_{\mathcal{C}\cap\mathcal{N}}(x^*),
\end{equation}
where $N_{\mathcal{C}\cap\mathcal{N}}(x^*)$ is the normal cone of $\mathcal{C}\cap\mathcal{N}$ on $x^*$ in $\mathcal{L}^2(\Omega,\mathscr{F}, R^n)$.

\end{definition}

\begin{remark}
The set $C(\o)$ represents the set of all available decisions  for each $\o\in \O$. The random vectors $\xi_{1},\xi_{2},...,\xi_{N-1}$ can be regarded as the observed information of the decision maker. In Definition \ref{Mstage SVI def}, the constraint $x\in \mathcal{C}$ means that, any admissible strategy $x$ should be valued in the decision set $C(\o)$ for $\mbox{a.s.}$ $\o\in \O$, while the nonanticipativity constraint  $x\in \mathcal{N}$ means that the admissible strategy $x $ should be a function (or a reaction) of the observed information up to now, but not rely on the observed information in the future.
\end{remark}

The multistage stochastic variational inequality $\mbox{MSVI}(F,\mathcal{C}\cap\mathcal{N} )$ \eqref{SVIbasicdefinition} is closely related to the following multistage stochastic variational inequality in extensive form.

\begin{definition}\label{Mstage SVI def extend}
Let $(\Omega,\mathscr{F},P)$ be a complete probability space, $F$ be a given map from $\mathcal{L}^2(\Omega,\mathscr{F},R^n)$ to $\mathcal{L}^2(\Omega,\mathscr{F}, R^n)$,  $\mathcal{N}$  and $\mathcal{C}$ be defined by \eqref{nonantipaspaceN}  and \eqref{constraintsetC} respectively.   The  multistage stochastic variational inequality in extensive form is:
To find $x^*\in\mathcal{C}\cap\mathcal{N}$ and $v^*\in \mathcal{M}$ such that
\begin{equation}\label{SVIextensivedefini}
-F(x^*)-v^*\in N_{\mathcal{C}}(x^*).
\end{equation}

\end{definition}

\begin{lemma}\label{appendix6}
Let $F:\mathcal{L}^2(\Omega,\mathscr{F},R^n)\to \mathcal{L}^2(\Omega,\mathscr{F},R^n)$, $\mathcal{C}$ be the nonempty closed convex set defined by (\ref{constraintsetC}). Then, for a random vector $x^*\in \mathcal{C}$, the condition
\begin{equation}\label{VI integral}
\dbE\langle F(x^*), x-x^*\rangle \geq0,\quad \forall~x\in\mathcal{C}
\end{equation}
is equivalent to the pointwise (for sample point) type condition
\begin{equation}\label{VI pointwise}
\langle F(x^*)(\omega), x-x^*(\omega)\rangle \geq0\quad \forall \ x \in C(\omega),\ \mbox{a.s.}~\omega\in\Omega.
\end{equation}
\end{lemma}

Lemma \ref{appendix6} has been proved in \cite{Rockafellar Wets17} in the discrete cases. For the convenience of readers, we provide a proof of Lemma \ref{appendix6} in the general probability space in the appendix.

By the definition of the normal cone of a closed convex set, the multistage stochastic variational inequality in extensive form (\ref{SVIextensivedefini}) can be rewritten as
\begin{equation}\label{SVIextensivedefini add}
\dbE\inner{F(x^*)+v^*}{x-x^*}\ge 0,\quad \forall\ x\in \mathcal{C}.
\end{equation}
By Lemma \ref{appendix6}, it is equivalent to finding  $x^*\in\mathcal{C}\cap\mathcal{N}$ and $v^*\in \mathcal{M}$ such that for a.s. $\o\in \O$,
\begin{equation}\label{SVIexten pointwise}
\inner{F(x^*)(\o)+v^*(\o)}{x-x^*(\o)}\ge 0,\quad \forall\ x\in C(\o).
\end{equation}

Since $N_{\mathcal{C}}(x^*)+N_{\mathcal{N}}(x^*)\subset N_{\mathcal{C} \cap\mathcal{N}}(x^*)$,  any solution to the  multistage stochastic variational inequality in extensive form \eqref{SVIextensivedefini} is a solution to $\mbox{MSVI}(F,\mathcal{C}\cap\mathcal{N} )$ \eqref{SVIbasicdefinition}.
When  the sum rule
\begin{equation}\label{sumpriciple_normcone}
N_{\mathcal{C} \cap\mathcal{N}}(x)=N_{\mathcal{C}}(x)+N_{\mathcal{N}}(x), \quad\ \forall \ x\in\mathcal{C} \cap\mathcal{N}
\end{equation}
is satisfied, \eqref{SVIbasicdefinition} is equivalent to \eqref{SVIextensivedefini}.

In what follows, we give a sufficient condition under which the sum rule \eqref{sumpriciple_normcone} holds true.

\begin{theorem}\label{th sum rule}
Let $(\Omega,\mathscr{F},P)$ be a complete probability space,   $\mathcal{N}$  and $\mathcal{C}$ be defined by \eqref{nonantipaspaceN}  and \eqref{constraintsetC} respectively. Suppose that for $i=0,1,...,N-1$, $C_{i}$ are $\mathscr{F}_{i}$-measurable. Then, the sum rule \eqref{sumpriciple_normcone} holds true.
\end{theorem}
\begin{proof}
Let $x\in \mathcal{C}\cap\mathcal{N}$ and define
\begin{equation}\label{sum rule eq1}
\mathcal{L}(x)\!:=\! \Big\{\xi=(\xi_{0},\xi_{1},...,\xi_{N-1})\!\in\! \mathcal{N}\ \Big|\ \xi_{i}(\o)\in N_{C_{i}(\o)}(x_{i}(\o)) \ \mbox{a.s.} \ \o\in \O,  \  \forall \ i=0,1,...,N-1  \Big\}.
\end{equation}
We claim that
\begin{equation*}
N_{\mathcal{C}\cap\mathcal{N}}(x)=\mathcal{M}+\mathcal{L}(x).
\end{equation*}

For any $\eta\in \mathcal{M}$, $\xi\in \mathcal{L}(x)$ and $y\in \mathcal{C}\cap\mathcal{N}$,
$$
\dbE\inner{\eta+\xi}{y-x}=\dbE\inner{\xi}{y-x}
=\sum_{i=0}^{N-1} \dbE\inner{\xi_{i}}{y_{i}-x_{i}}\le 0.
$$
It implies  $\mathcal{M}+\mathcal{L}(x)\subset N_{\mathcal{C}\cap\mathcal{N}}(x)$.

On the other hand, for any $\xi\in N_{\mathcal{C}\cap\mathcal{N}}(x)$ we have $$\xi= \Pi_{\mathcal{N}}(\xi)+\Pi_{\mathcal{M}}(\xi),$$
where $\Pi_{\mathcal{N}}(\xi)$ and $\Pi_{\mathcal{M}}(\xi)$ are respectively the projections of $\xi$ onto $\mathcal{N}$ and $\mathcal{M}$. Clearly $\Pi_{\mathcal{M}}(\xi)\in \mathcal{M}$, $\Pi_{\mathcal{N}}(\xi)\in\mathcal{N}$.  By  $\xi\in N_{\mathcal{C}\cap\mathcal{N}}(x)$, we have
$$
0\ge\dbE\inner{\xi}{y-x}
=\dbE\inner{\Pi_{\mathcal{N}}(\xi)+\Pi_{\mathcal{M}}(\xi)}{y-x}
=\dbE\inner{\Pi_{\mathcal{N}}(\xi)}{y-x},\ \forall  y\in \mathcal{C}\cap\mathcal{N}.
$$
For any fixed $i=0,1,...,N-1$ and any $y_{i}\in \mathcal{L}^2(\Omega,\mathscr{F}_{i}, R^{n_i})$ such that $y_{i}(\o)\in C_{i}(\o)$  a.s. $\o\in \O$, we define $\tilde y=(x_0,x_1,...,x_{i-1},y_i,x_{i+1},..., x_{N-1})$. Then, $\tilde y\in \mathcal{C}\cap\mathcal{N}$ and
\begin{equation}\label{eq 2.15+}
0\ge\dbE\inner{\Pi_{\mathcal{N}}(\xi)}{\tilde y-x}=\dbE\inner{(\Pi_{\mathcal{N}}(\xi))_{i}}{y_{i}-x_{i}}.
\end{equation}
Here, $(\Pi_{\mathcal{N}}(\xi))_{i}$ is the $i$th component of $\Pi_{\mathcal{N}}(\xi)$.  Clearly, $(\Pi_{\mathcal{N}}(\xi))_{0}\in N_{C_0}(x_0)$. For any  $i=1,2,...,N-1$, by the definition of $\mathscr{F}_{i}$, $(\Omega,\mathscr{F}_{i},P)$ is a complete probability space.  Then, by \eqref{eq 2.15+}, the $\mathscr{F}_{i}$-measurability of $C_{i}$  and a similar proof of Lemma \ref{appendix6}, we conclude that
$$(\Pi_{\mathcal{N}}(\xi))_{i}(\o)\in N_{C_{i}(\o)}(x_{i}(\o))\ \mbox{a.s.}\ \o\in \O.$$
Therefore, $\Pi_{\mathcal{N}}(\xi)\in \mathcal{L}(x)$. Consequently, $N_{\mathcal{C}\cap\mathcal{N}}(x)\subset \mathcal{M}+\mathcal{L}(x)$. This proves $N_{\mathcal{C}\cap\mathcal{N}}(x)=\mathcal{M}+\mathcal{L}(x)$.

Since
$$N_{\mathcal{C}}(x)+N_{\mathcal{N}}(x)=N_{\mathcal{C}}(x)+ \mathcal{M}\subset N_{\mathcal{C}\cap\mathcal{N}}(x),$$
to prove \eqref{sumpriciple_normcone}, we only need to prove that $ \mathcal{L}(x)\subset N_{\mathcal{C}}(x)$. Using a similar proof of Lemma \ref{appendix6} again we deduce that
\begin{equation}\label{sum rule eq2}
N_{\mathcal{C}}(x)= \Big\{\xi\in \mathcal{L}^2(\Omega,\mathscr{F}, R^{n}) \ \Big|\ \xi_{i}(\o)\in N_{C_{i}(\o)}(x_{i}(\o))\  \mbox{a.s.}\ \o\in\O, \  \forall \ i=0,1,...,N-1 \Big\}.
\end{equation}
Then, the conclusion follows immediately from the definition of $\mathcal{L}(x)$.
\end{proof}

\begin{remark}
Let $\mathcal{H}$ be a Hilbert space,  $\mathcal{S}_1, \mathcal{S}_2\subset \mathcal{H}$ be nonempty closed convex subsets of $\mathcal{H}$.
When $\mathcal{H}$ is a finite dimensional space, by \cite[Corollary
23.8.1]{Rockafellar 1970},
\begin{equation}\label{eq sumpriciple}
N_{\mathcal{S}_1\cap\mathcal{S}_2}(x)=N_{\mathcal{S}_1}( x)+N_{\mathcal{S}_2}( x)
\end{equation}
holds true if $ri(\mathcal{S}_1)\cap ri(\mathcal{S}_2) \neq \emptyset$, where $ri(\mathcal{S}_1)$ and $ri(\mathcal{S}_2)$ are the relative interiors of $\mathcal{S}_1$ and $\mathcal{S}_2$, respectively.  However, in the general infinite dimensional cases,  the condition $ri(\mathcal{S}_1)\cap ri(\mathcal{S}_2) \neq \emptyset$ is not enough to ensure the sum rule \eqref{eq sumpriciple}. A counterexample can be found in \cite{Borwein Goebel 2003}. In the infinite dimensional cases, the sum rule \eqref{eq sumpriciple} holds true if
\begin{equation}\label{Att-Brezis CQ}
\text{cone}(\mathcal{S}_1-\mathcal{S}_2) \text{ is a closed linear subspace of } \mathcal{H}.
\end{equation}
Here, $\text{cone}(\mathcal{S}_1-\mathcal{S}_2)$ is the cone generated by $\mathcal{S}_1-\mathcal{S}_2$. \eqref{Att-Brezis CQ} is called the Attouch-Brezis qualification condition and a proof the sum rule \eqref{eq sumpriciple} under condition \eqref{Att-Brezis CQ} can be found in \cite{Mordukhovich2022}.

Note that $\mathcal{N}$ is a closed linear subspace of $\mathcal{L}^2( \O,\mathscr{F},\mrn)$. When the sample space $\O$ is a finite set,  $\mathcal{L}^2( \O,\mathscr{F},\mrn)$ is isomorphic to a finite dimensional Euclidean space (see \cite{Rockafellar Wets17} or Section 4 of this paper for more details). In this case, the equality \eqref{sumpriciple_normcone} holds true if  $ri(\mathcal{C})\cap \mathcal{N}\neq\emptyset$.  It is proved in \cite[Theorem 2.3]{Rockafellar Wets17} that, when the sample space $\O$ is a finite set, $ri(\mathcal{C})\cap \mathcal{N}\neq\emptyset$ if there is $\hat x\in \mathcal{N}$ such that $\hat x(\omega)\in ri(C(\omega))$ for all $\o\in \O$.
In the general cases, $\mathcal{L}^2( \O,\mathscr{F},\mrn)$ is an infinite dimensional Hilbert space. In order that the sum rule \eqref{sumpriciple_normcone} holds true, it seems a natural way to assume the Attouch-Brezis qualification condition for $\mathcal{C}$ and $\mathcal{N}$.
However, in the stochastic cases, the Attouch-Brezis qualification condition may fail even when $C(\o)$ is a polyhedral for any $\o\in \O$.

Let us consider a simple example. Let $\O=[0,1]$, $\mathscr{F}$ be the $\sigma$-field of Lebesgue  measurable sets and $P$ be the Lebesgue  measure on $[0,1]$. Clearly, $(\Omega,\mathscr{F},P)$ is a complete probability space. We consider the special case of one-stage with $C(\o)\equiv [0,+\infty)$. Then,   $\mathcal{C}=\{x\in L^{2}(0,1)\ |\  x(\o)\ge 0,\ a.e.\ \o\in [0,1]\}$ and  $\mathcal{N}=\{x\in L^{2}(0,1)\ |\  x(\o)= a \  a.e.\ \o\in [0,1], a\in R \}$.
By Theorem \ref{th sum rule}, $N_{\mathcal{C}\cap\mathcal{N}}(x)=N_{\mathcal{C}}(x)+N_{\mathcal{N}}(x)$ for any $x\in \mathcal{C}\cap\mathcal{N}$. For any $x\in L^{\infty}(0,1)$,
$$x(\o)=x(\o)+\|x\|_{L^{\infty}}-\|x\|_{L^{\infty}},\quad a.e.\ \o\in [0,1],$$
where $\|x\|_{L^{\infty}}$ is the $L^{\infty}$ norm of $x$.
Letting $\tilde x(\o)=x(\o)+\|x\|_{L^{\infty}}$, a.e. $\o\in [0,1]$, we have $\tilde x\in \mathcal{C}$ and
$$x=\tilde x- \|x\|_{L^{\infty}}\in \mathcal{C}-\mathcal{N}.$$
It follows that $L^{\infty}(0,1)\subset  \text{cone}(\mathcal{C}-\mathcal{N})$. Therefore, if $\text{cone}(\mathcal{C}-\mathcal{N})$ is a closed linear subspace of $L^{2}(0,1)$, it must be $L^{2}(0,1)$. However, for any $x\in L^{2}(0,1)\setminus L^{\infty}(0,1)$ such that $x(\o)< 0$ a.e. $\o\in(0,1)$, $x\notin \text{cone}(\mathcal{C}-\mathcal{N})$. Therefore, the  Attouch-Brezis qualification condition dose not hold true in this example.
\end{remark}

As illustrated in \cite{Rockafellar Wets17}, one of the motivations for studying $\mbox{MSVI}(F,\mathcal{C}\cap\mathcal{N} )$ \eqref{SVIbasicdefinition} is to solve the multistage stochastic convex optimization problem: To find  $x^*\in\mathcal{C}\cap\mathcal{N}$ such that
\begin{equation*}
f(x^*)=\operatorname*{min}_{x\in\mathcal{C}\cap\mathcal{N}}f(x),
\end{equation*}
where $f(x)=\dbE\, g(x(\o),\o)$ and $g:R^n\times \O\rightarrow R$ is a given function which is continuously differentiable and convex with respect to the first variable and $\mathscr{F}$-measurable with respect to the second variable.  If $\dbE\, |g(0,\o)|<+\infty$ and there is an $\mathscr{F}$-measurable nonnegative random variable $\eta$ with $\dbE\, \eta^2(\o)<+\infty$ such that
\begin{equation}\label{eq 2.18}
|g(u,\o)-g(v,\o)|\leq \eta(\o)|u-v|\ \mbox{a.s.}\ \o\in\O,\quad \forall\ u,v\in \mrn,
\end{equation}
then, the function $f: \mathcal{L}^2(\Omega,\mathscr{F}, R^n)\to R$ is a well-defined  differentiable convex function on $\mathcal{L}^2(\Omega,\mathscr{F}, R^n)$  with its G\^{a}teaux derivative
\begin{equation}\label{expect nabla}
 D f(x)=\nabla_{x} g(x(\o),\o)\ \mbox{a.s.}\ \o\in\O,\quad \forall\ x\in \mathcal{L}^2(\Omega,\mathscr{F}, R^n).
\end{equation}

Indeed, for any $x\in \mathcal{L}^2(\Omega,\mathscr{F}, R^n)$,
$$|f(x)|=|\dbE\, g(x(\o),\o)|\le\dbE\, |g(x(\o),\o)-g(0,\o)|+\dbE\, |g(0,\o)|\le\dbE\, \eta(\o)|x(\o)|+\dbE\, |g(0,\o)|<+\infty.$$
This proves that $f$ is well-defined on $\mathcal{L}^2(\Omega,\mathscr{F}, R^n)$.
For any $d\in \mathcal{L}^2(\Omega,\mathscr{F}, R^n)$, by \eqref{eq 2.18}  we have
\begin{equation*}
\Big|\frac{ g(x(\o)+td(\o),\o)-g(x(\o),\o)}{t}\Big|\le \eta(\o)|d(\o)|  \ \mbox{a.s. } \o\in\O
\end{equation*}
and $\dbE\, \eta(\o)|d(\o)|<+\infty$.
By Lebesgue Dominated Convergence Theorem, we have
\begin{eqnarray*}
\lim_{t\to 0^+}\frac{f(x+td)-f(x)}{t}&=&\lim_{t\to 0^+}\dbE\Big[\frac{ g(x(\o)+td(\o),\o)-g(x(\o),\o)}{t}\Big]\\
&=&\dbE\inner{\nabla_{x} g(x(\o),\o)}{d(\o)}.
\end{eqnarray*}
Then, by Riesz Representation Theorem, we have \eqref{expect nabla}.

By the basic theory of convex optimization, solving the above multistage stochastic convex optimization problem is equivalent to solving the $\mbox{MSVI}(F,\mathcal{C}\cap\mathcal{N} )$ \eqref{SVIbasicdefinition} with $F(x)=Df(x)$.

To end this section, we provide an application of $\mbox{MSVI}(F,\mathcal{C}\cap\mathcal{N} )$ \eqref{SVIbasicdefinition} in stochastic optimal control problem. We refer the reader to \cite{YongZhou1999} for some basic notions in stochastic control theory.

\begin{example}\label{ex SOCP}
Let $(\Omega, \mathscr{F},  \mathbb{F}, P)$  be a complete filtered probability space with the filtration $\mathbb{F}=\{\mathscr{F}_{t} \}_{0\le t\le 1}$, on which a one-dimensional standard  Wiener
process $W(\cdot)$ is defined such that $\mathbb{F}$ is the natural filtration generated by $W(\cdot)$ (augmented by all the P-null sets), i.e., for any $t\in [0,1]$,
$$\mathscr{F}_{t}=\sigma(\mathscr{F}_{t}^{W}\cup\mathscr{O})$$
with $\mathscr{F}_{t}^{W}:=\sigma(W(s): s\in [0,t])$ and $\mathscr{O}$ being the collection of all $P$-null sets.

Let $A,\tilde A:[0,1]\to R^{n\times n}$ and $B,D:[0,1]\to R^{n\times m}$ be given bounded Borel measurable matrix-valued maps, $G\in R^{n\times n}$ be a given positive semi-definite matrix and $U\subset R^m$ be a given closed convex set. Consider the controlled linear stochastic differential equation
\begin{align}\label{eq controlsys}
\left\{
\begin{array}{l}	d x(t)=\big(A(t)x(t)+B(t)u(t)\big)dt+\big(\tilde A(t)x(t)+D(t)u(t)\big)dW(t), \quad t\in[0,1]\\
x(0)=x_0
\end{array}
\right.
\end{align}
with the Mayer type cost functional
\begin{equation}\label{costfun}
J(u)=\frac{1}{2}\dbE\langle G(x(1)-\eta),x(1)-\eta\rangle.
\end{equation}
Here $u\in \mathcal{U}$ is called the control,
\begin{eqnarray*}
\mathcal{U}\!\!\!&:=&\!\!\Big\{u\in L^{2}_{\mathbb{F}}(0,1;R^m)\ \Big| u(t)\in U,\ a.e.\ t\in[0,1], \ \mbox{a.s.}   \Big\},
\end{eqnarray*}
$L^{2}_{\mathbb{F}}(0,1;R^m)$ is the space of $R^m$-valued $\mathscr{B}([0,1])\otimes\mathscr{F}$-measurable  stochastic processes $\varphi$ such that for any $t\in [0,1]$, $\varphi(t)$ is $\mathscr{F}_t$-measurable  and $\big[\mathbb{E}\int_{0}^{1}|\varphi(t)|^2dt\big]^{\frac{1}{2}}<\infty$.
$x$ is the state valued in $R^n$ with initial datum $x_0\in R^n$ and control $u$, and, $\eta:\O\to R^n$ is an $\mathscr{F}_{1}$-measurable random vector. $x$ is a solution to \eqref{eq controlsys} if
$$x(t)=x_0+ \int_{0}^{t}\big(A(s)x(s)+B(s)u(s)\big)ds+  \int_{0}^{t}\big(\tilde A(s)x(s)+D(s)u(s)\big)dW(s) \ \mbox{a.s.},\ \forall\ t\in [0,1].$$
Here, $\int_{0}^{t}\varphi(s)dW(s)$ is the It\^{o}'s stochastic integral of stochastic process $\varphi$ on $[0,t]$ with respect to Wiener process $W$. In order that the It\^{o}'s stochastic integral is well-defined, we usually assume $\varphi$ is a square-integrable stochastic process and  for any $t\in [0,1]$, $\varphi(t)$ is $\mathscr{F}_t$-measurable.
By the standard theory of stochastic analysis (see for instance \cite[Chapter 1]{YongZhou1999}),  for any $x_0\in R^n$ and $u\in \mathcal{U}$, \eqref{eq controlsys} admits unique solution. We consider the following
stochastic optimal control problem with control constraints: To find $u^*\in \mathcal{U}$ such that
\begin{equation}\label{SOCP}
J(u^*)=	\min_{u\in \mathcal{U}}  J(u).
\end{equation}

The continuous-time stochastic optimal control problem \eqref{SOCP} has important applications in mathematical finance (see, for instance \cite{LiZhouLim2002}) and can be regarded as a stochastic convex programming problem defined on the Hilbert space of square-integrable stochastic processes. Its discretized  approximation problem is a multistage stochastic convex  programming problem with nonanticipativity constraint. To see this, let $N$ be a given large enough natural number, $\Delta=1/N$, $\Delta W_{i}=W((i+1)/N)-W(i/N)$, $i=0,1,...,N-1$ and consider the Euler  approximation of the stochastic differential equation \eqref{eq controlsys}:
$$\left\{
\begin{aligned}
&x_{i+1}=x_{i}+[A_{i}x_{i}+B_i u_i]\Delta+[\tilde A_i x_{i}+D_i u_i]\Delta W_i,\ i=0,...,N-1,\\[-0.2em]
&x_0\in R^n.
\end{aligned}
\right.$$
Here, $A_i=A(i/N)$, $B_i=B(i/N)$, $\tilde A_i=\tilde A(i/N)$, $D_i=D(i/N)$, $\mathscr{U}^{N}:=(u_0,u_1,...,u_{N-1}):\O\to (R^{m})^{N}$ is the discretized control and $\mathscr{X}^{N}:=(x_0,x_1,...,x_N):\O\to (R^{n})^{N+1}$ is the corresponding discretized state. Denote $\Psi_{i}=I+A_{i}\Delta+\tilde A_{i}\Delta W_i$, $\Lambda_{i}=B_{i}\Delta+D_{i}\Delta W_i$, $i=0,1,...,N-1$. We have
\begin{equation}\label{discr stat X1}
x_{N}=\Big[\prod_{i=0}^{N-1}\Psi_{i}\Big]x_0+\sum_{i=0}^{N-1} \Big[\prod_{j=i+1}^{N-1}\Psi_{j}\Big]\Lambda_{i}u_i.
\end{equation}
Then, the corresponding  discretized cost function is represented by
\begin{equation}\label{disc cotsfun}
J^{N}(\mathscr{U}^{N})
\!=\!\frac{1}{2}\dbE\inner{\!G\Bigg\{\Big[\prod_{i=0}^{N-1}\Psi_{i}\Big]x_0\!+\!\sum_{i=0}^{N-1} \! \Big[\!\prod_{j=i+1}^{N-1}\Psi_{j}\Big]\Lambda_{i}u_i-\eta\Bigg\}}{\Big[\!\prod_{i=0}^{N-1}\Psi_{i}\Big]x_0\!+\!\sum_{i=0}^{N-1}\! \Big[\!\prod_{j=i+1}^{N-1}\Psi_{j}\Big]\Lambda_{i}u_i-\eta\!}.
\end{equation}

By the definition of  standard Wiener process, $\Delta W_i: \O\to R$, $k=0,1,...,N-1$ is a sequence of  mutually independent Gaussian random variables with expectation $0$ and variance $1/N$. Let $\mathscr{F}_0=\{\emptyset, \O\}$, $$\mathscr{F}_{i}=\sigma\Big\{\sigma(\Delta W_0, \Delta W_1,...,\Delta W_{i-1})\cup \mathscr{O}\Big\} ,\   i=1,2,...,N-1$$ and define
\begin{equation*}
\mathcal{N}:=\Big\{\mathscr{U}^{N}=(u_{0},u_{1},...,u_{N-1})\in \mathcal{L}^2(\Omega,\mathscr{F}, (R^m)^{N}) \ \big|\ u_{i}\in \mathcal{L}^0(\Omega, \mathscr{F}_{i}, R^{m}), i=0,1,...,N-1\Big\}.
\end{equation*}
In addition, we define
\begin{equation*}
\mathcal{C}:=\Big\{\mathscr{U}^{N}=(u_{0},u_{1},...,u_{N-1})\in \mathcal{L}^2(\Omega,\mathscr{F}, (R^m)^{N})\ \Big|\ u_{i}(\omega)\in U\ \mbox{a.s. } \o\in\O,\ i= 0,1,\cdots,N-1 \Big\}.
\end{equation*}
The discretized approximation problem for the continuous-time stochastic optimal control problem \eqref{SOCP} is: To find $(\mathscr{U}^{N})^*\in \mathcal{C}\cap \mathcal{N}$ such that
\begin{equation}\label{disc SOCP}
J^{N}((\mathscr{U}^{N})^*)=\min_{\mathscr{U}^{N}\in \mathcal{C}\cap \mathcal{N}}J^{N}(\mathscr{U}^{N}).
\end{equation}
Clearly, \eqref{disc SOCP} is a multistage stochastic convex programming problem with nonanticipativity constraint.

Denote
$$ \zeta=\Big[\prod_{i=0}^{N-1}\Psi_{i}\Big]x_0,\ Z=\Bigg(\Big[\prod_{j=1}^{N-1}\Psi_{j}\Big]\Lambda_{0},
\Big[\prod_{j=2}^{N-1}\Psi_{j}\Big]\Lambda_{1},..., \Lambda_{N-1}\Bigg)$$
and let $b=Z^{\top}G(\zeta-\eta)$ and $M=Z^{\top}GZ$.
The derivative of $J^N$ with respect to $\mathscr{U}^{N}$ is
\begin{equation}\label{disc derivetive}
DJ^{N}(\mathscr{U}^{N})=M\mathscr{U}^{N}+b
\end{equation}
and the first order necessary condition for the stochastic convex programming problem \eqref{disc SOCP} is
$$-DJ^{N}((\mathscr{U}^{N})^*)\in N_{\mathcal{C}\cap \mathcal{N}}((\mathscr{U}^{N})^*), $$
which is a multistage stochastic variational inequality  with nonanticipativity constraint.
\end{example}

\section{Algorithm and convergence analysis}\label{sec3}

In this section, we shall propose the prediction-correction ADMM for solving $\mbox{MSVI}(F,\mathcal{C}\cap\mathcal{N} )$ (\ref{SVIbasicdefinition}) and prove that the  sequence generated by that  algorithm converges weakly in $\mathcal{L}^2(\Omega,\mathscr{F}, R^n)$ to a solution to $\mbox{MSVI}(F,\mathcal{C}\cap\mathcal{N} )$ (\ref{SVIbasicdefinition}).

We assume that the map $F:\mathcal{L}^2(\Omega,\mathscr{F}, R^n)\to \mathcal{L}^2(\Omega,\mathscr{F}, R^n)$ satisfies the following conditions.
\begin{enumerate}
\item [(A1)]  The map $F$ is monotone.
\item [(A2)] The map $F$ is Lipschitz continuous with constant $L_{F}>0$, i.e.,
$$\|F(x)-F(y)\|_{\mathcal{L}^2}\le L_{F}\|x-y\|_{\mathcal{L}^2},\quad \forall \ x,y\in \mathcal{L}^2(\Omega,\mathscr{F}, R^n).$$
\end{enumerate}

The following is a simple example in which the above assumptions are satisfied.

\begin{example}
Let $f(x)=\dbE\, g(x(\o),\o)$ and $g:R^n\times \O\rightarrow R$ be a given function which is continuously differentiable and convex with respect to the first variable. Assume $\dbE\, |g(0,\o)|<+\infty$ and there is an $\mathscr{F}$-measurable nonnegative random variable $\eta$ with $\dbE\, \eta^2(\o)<+\infty$ such that  \eqref{eq 2.18} holds true. Moreover, assume that there is a  constant $L>0$ such that
\begin{equation}\label{eq 3.1+}
|\nabla_{x}g(u,\o)-\nabla_{x}g(v,\o)|\le L|u-v|\ \mbox{a.s. } \o\in\O,\ \forall \ u,v\in \mrn,
\end{equation}
then $F(\cdot)=Df(\cdot)$ is monotone and Lipschitz continuous on $\mathcal{L}^2(\Omega,\mathscr{F}, R^n)$.  Actually, by the convexity of $g$ with respect to the first variable, we have, for any $x,y\in \mathcal{L}^2(\Omega,\mathscr{F}, R^n)$,
$$\inner{\nabla_{x}g(x(\o),\o)-\nabla_{x}g(y(\o),\o)}{x(\o)-y(\o)}\ge 0  \ \mbox{a.s. }\o\in\O$$
which implies that
$$\dbE\inner{F(x)-F(y)}{x-y}=\dbE\inner{\nabla_{x}g(x(\o),\o)-\nabla_{x}g(y(\o),\o)}{x(\o)-y(\o)}\ge 0.$$
In addition, by \eqref{eq 3.1+}, for any $x,y\in \mathcal{L}^2(\Omega,\mathscr{F}, R^n)$,
$$\|F(x)-F(y)\|_{\mathcal{L}^2}
=\big[\dbE|\nabla_{x}g(x(\o),\o)-\nabla_{x}g(y(\o),\o)|^2\big]^{\frac{1}{2}}
\le L\big[\dbE|x-y|^2\big]^{\frac{1}{2}}=L\|x-y\|_{\mathcal{L}^2}.$$

\end{example}

\subsection{Prediction-correction ADMM}

In this subsection, we will introduce the prediction-correction ADMM for $\mbox{MSVI}(F,\mathcal{C}\cap\mathcal{N} )$ (\ref{SVIbasicdefinition}).
In order to illustrate the key idea of this algorithm clearly, let us go back for a moment  to the multistage stochastic convex optimization problem: To find  $x^*\in \mathcal{C}\cap\mathcal{N}$ such that
\begin{equation}\label{miniexpectfunction}
f(x^*)=\operatorname*{min}_{x\in\mathcal{C}\cap\mathcal{N}}f(x)
\end{equation}
with $f(x)=\dbE\, g(x(\o),\o)$ and $g:R^n\times \O\rightarrow R$.
Clearly,  problem (\ref{miniexpectfunction})  is equivalent to the following optimization problem with a separable objective function and a linear equality constraint:
\begin{equation}\label{StocOptiProConstraint}
\text{min}\Big\{f(x)+I_{\mathcal{N}}(y)\ \Big|\  x \in \mathcal{C}, y\in \mathcal{L}^2(\Omega,\mathscr{F}, R^n), x-y=0\Big\},
\end{equation}
where $I_{\mathcal{N}}(\cdot)$ is the indicator function of the nonanticipativity subspace $\mathcal{N}$  defined by
\begin{equation*}
I_{\mathcal{N}}(y):=\left\{
\begin{aligned}
&0,\quad \quad\quad\quad\text{if}~ y \in \mathcal{N};\\
&+\infty, \quad\quad~\text{if}~ y \notin \mathcal{N}.
\end{aligned}
\right.
\end{equation*}

Denote $\big[\mathcal{L}^2(\Omega,\mathscr{F}, R^n)\big]^3=\mathcal{L}^2(\Omega,\mathscr{F}, R^n)\times \mathcal{L}^2(\Omega,\mathscr{F}, R^n)\times \mathcal{L}^2(\Omega,\mathscr{F}, R^n)$.
Let $\beta > 0$ and define the  augmented Lagrange function for problem  (\ref{StocOptiProConstraint}) by
\begin{equation}\label{Lagrangefunction}
L_{\beta}(x,y,\lambda)=
f(x)+I_{\mathcal{N}}(y) -\dbE \langle\lambda,x-y\rangle +\frac{\beta}{2}\dbE |x-y|^2,
\end{equation}
for any $(x,y,\lambda)^{\top}\in \big[\mathcal{L}^2(\Omega,\mathscr{F}, R^n)\big]^3$.

Given a triplet $(x^k,y^k,\lambda^k)^{\top}\in \mathcal{C}\times \mathcal{N}\times \mathcal{L}^2(\Omega,\mathscr{F}, R^n)$, the classical ADMM  generates a new iterate $(x^{k+1},y^{k+1},\lambda^{k+1})^{\top}$ via the following  procedure
\begin{equation}\label{ADMMstochaprogram}
\begin{cases}
x^{k+1}=\operatorname*{argmin}\limits_{x\in\mathcal{C}}
 \big\{ f(x)-\dbE~ \langle\lambda^{k},x-y^{k}\rangle +\frac{\beta}{2}\dbE~ |x-y^k|^2\big\},\\[+0.5em]
y^{k+1}=\operatorname*{argmin}\limits_{y\in\mathcal{L}^2(\Omega,\mathscr{F}, R^n)}
 \big\{I_{\mathcal{N}}(y)-\dbE~\langle\lambda^{k},x^{k+1}-y\rangle
+\frac{\beta}{2}\dbE~ |x^{k+1}-y|^2\big\},\\[+0.5em]
\lambda^{k+1}=\lambda^{k}-\beta(x^{k+1}-y^{k+1}).
\end{cases}
\end{equation}
By the first order necessary condition in optimization,  we have $(x^{k+1},y^{k+1},\lambda^{k+1})^{\top}$ satisfies
\begin{equation}\label{ADMM inclusion}
\left\{
\begin{array}{l}
-D f(x^{k+1})+ \lambda^{k}-\beta(x^{k+1}-y^k)\in N_{\mathcal{C}}(x^{k+1}),\\[+0.5em]
-\lambda^{k}+\beta(x^{k+1}-y^{k+1})\in N_{\mathcal{N}}(y^{k+1}),\\[+0.5em]
\lambda^{k+1}=\lambda^{k}-\beta(x^{k+1}-y^{k+1}).
\end{array}\right.
\end{equation}
By $Df(x^{k+1})=\nabla_{x} g(x^{k+1}(\o),\o)$, a.s. $\o\in\O$ and  Lemma \ref{appendix6}, the first variational inclusion in (\ref{ADMM inclusion}) is equivalent to the following (sample point) pointwise version
$$
-\nabla_{x} g(x^{k+1}(\o),\o)+ \lambda^k(\o)-\beta(x^{k+1}(\o)-y^k(\o))\in N_{C(\o)}(x^{k+1}(\o)) \ \mbox{a.s. } \o\in\O.$$
Letting $\hat u^k:=x^{k+1}$, $u^{k+1}:=y^{k+1}$, $v^{k}:=-\lambda^{k}$,
the second variational inclusion in (\ref{ADMM inclusion}) can be rewritten equivalently as $u^{k+1}=\Pi_{\mathcal{N}}\big(\hat{u}^{k}+\frac{1}{\beta}v^k\big)$. In addition, if we choose $v^{0}:=-\lambda^{0}\in \mathcal{M}$, then we can obtain by induction that $v^{k}\in \mathcal{M}$,
$$u^{k+1}= \Pi_{\mathcal{N}}\big(\hat{u}^{k}+\frac{1}{\beta}v^k\big) =\Pi_{\mathcal{N}}\big(\hat{u}^{k}\big)+\frac{1}{\beta}\Pi_{\mathcal{N}}\big(v^k\big)=\Pi_{\mathcal{N}}\big(\hat{u}^{k}\big)$$
and $v^{k+1}=v^{k}+\beta(\hat{u}^{k}-\Pi_{\mathcal{N}}\big(\hat{u}^{k}\big)) =v^{k}+\beta\Pi_\mathcal{M}(\hat u^k)$. Now, taking $F(\cdot)=D f(\cdot)$, \eqref{ADMM inclusion} becomes
\begin{equation}\label{PHA}
\left\{
\begin{array}{l}
-F(\hat u^k)(\o)-v^k(\o)-\beta(\hat u^k(\o)-u^k(\o))\in N_{C(\o)}(\hat u^k(\o))\ \mbox{a.s. } \o\in\O,\\[+0.4em]
u^{k+1}=\Pi_{\mathcal{N}}\big(\hat{u}^{k}\big),\\[+0.4em]
v^{k+1}=v^{k}+\beta\Pi_\mathcal{M}(\hat u^k),
\end{array}\right.
\end{equation}
which coincides with the progressive hedging algorithm in \cite{Rockafellar Sun 2019}.

Note that in  (\ref{ADMM inclusion}) (resp. (\ref{PHA})), to obtain $x^{k+1}$  (resp. $\hat u^k$) one has to solve a collection of parameterized (with parameter $\o\in \O$) finite dimensional variational inequalities. When the structure of those variational inequalities are complicated (for example, high nonlinearity of $F(\cdot)$ or $D f(\cdot)$), solving those finite dimensional variational inequalities will become time-consuming.

Taking $F(x)=D f(x)$, by the  properties of metric projection,  (\ref{ADMM inclusion}) can be rewritten   as
\begin{equation}\label{ADMM projection}
\left\{
\begin{array}{l}
x^{k+1}=\Pi_{\mathcal{C}}\big\{x^k-\frac{1}{\beta}[F(x^{k+1})-\lambda^k
+\beta(x^k-y^k)]\big\},\\[+0.5em]
y^{k+1}=\Pi_{\mathcal{N}}\big\{y^k-\frac{1}{\beta}[\lambda^{k}-\beta(x^{k+1}-y^k)]\big\};\\[+0.5em]
\lambda^{k+1}=\lambda^k-\beta(x^{k+1}-y^{k+1}).
\end{array}\right.
\end{equation}
To obtain an explicit iterative algorithm, one natural way is to replace  $x^{k+1}$ by $x^{k}$ directly on the right side of the first equation in (\ref{ADMM projection}). However, as pointed out in \cite{HeJCM2006} for deterministic structured variational inequalities, such modification might lead to divergence and some proper correction is necessary to ensure the convergence. Here we shall adapt the prediction-correction ADMM which was proposed in \cite{HeJCM2006} to solve the multistage stochastic variational inequality $\mbox{MSVI}(F,\mathcal{C}\cap\mathcal{N} )$ (\ref{SVIbasicdefinition}).  Before giving the algorithm, we introduce some useful notations which will be used in the sequel.

Let $\beta, r >0$ and define $G$ from $\big[\mathcal{L}^2(\Omega,\mathscr{F}, R^n)\big]^3$ to  $\big[\mathcal{L}^2(\Omega,\mathscr{F}, R^n)\big]^3$
by
\begin{equation}\label{G}
G(\theta)
=\begin{pmatrix}
\beta rx\\
\beta y\\
\frac{1}{\beta}\lambda
\end{pmatrix},\quad
\forall \theta=\begin{pmatrix}
 x\\
y\\
\lambda
\end{pmatrix}\in\big[\mathcal{L}^2(\Omega,\mathscr{F}, R^n)\big]^3.
\end{equation}
Clearly, $G$ is a symmetry operator and $G^{-1}$ exists with
\begin{equation*}
G^{-1}(\vartheta)=\begin{pmatrix}
\frac{1}{\beta r}u\\
\frac{1}{\beta} v\\
 \beta w
\end{pmatrix},\quad\forall\
\vartheta=\begin{pmatrix}
u\\
v\\
w
\end{pmatrix}\in \big[\mathcal{L}^2(\Omega,\mathscr{F}, R^n)\big]^3.
\end{equation*}
Let   $\theta^k=(x^k,y^k,\lambda^k)^{\top}$, $\tilde \theta^k=(\tilde{x}^k,\tilde{y}^k,\tilde{\lambda}^k)^{\top}\in \big[\mathcal{L}^2(\Omega,\mathscr{F}, R^n)\big]^3$ ($k\in \mathbb{N}$), define
\begin{equation}\label{SVIxixk}
\zeta_{x}^k:=F(x^k)
-F(\tilde{x}^k)+\beta(x^k-\tilde{x}^k),
\end{equation}
and denote
\begin{equation*}
\zeta^k:=\begin{pmatrix}
\zeta^{k}_{x}\\
0\\
0
\end{pmatrix}
\end{equation*}
and
\begin{equation}\label{theaktildexik}
d(\theta^k,\tilde{\theta}^{k},\zeta^k):=\theta^k-\tilde{\theta}^k-G^{-1}\zeta^k.
\end{equation}

The prediction-correction ADMM  for $\mbox{MSVI}(F,\mathcal{C}\cap\mathcal{N} )$ (\ref{SVIbasicdefinition}) is defined as follows.

\bigskip
\begin{algorithm}
\caption{\bf Algorithm 3.1: Prediction-correction ADMM}\label{algo1}
\textbf{Step 0.}
Set  $\beta, r >0$, $r>\frac{L_{F}}{\beta}+1$, $\alpha\in (0,1)$, $\theta^0=(x^0, y^0,\lambda^0)^{\top}\in \mathcal{C}\times\mathcal{N}\times\mathcal{L}^2(\Omega,\mathscr{F}, R^n)$,   $k=0$.

\textbf{Step 1: Prediction}

\hspace{+2em}Step 1.1. Compute
\begin{equation}\label{precoreSVIx}
\tilde{x}^{k}=\Pi_{\mathcal{C}}\{x^k-\frac{1}{\beta r}[F(x^{k})-\lambda^k+\beta(x^k-y^k)]\}.
\end{equation}

\hspace{+2em}Step 1.2. Compute
\begin{equation}\label{precoreSVIy}
\tilde{y}^{k}=\Pi_{\mathcal{N}}\{y^k-\frac{1}{\beta}
[\lambda^k-\beta(\tilde{x}^{k}-y^k)]\}.
\end{equation}

\hspace{+2em}Step 1.3. Update $\tilde{\lambda}^k$ via
\begin{equation}\label{precoreSVIlambda}
\tilde{\lambda}^{k}=\lambda^k-\beta(\tilde{x}^{k}-\tilde{y}^{k}).
\end{equation}

\textbf{Step 2: Correction}

Set
\begin{equation}\label{thetak+1thetak}
\theta^{k+1}=\theta^{k}-\alpha d(\theta^{k},\tilde{\theta}^{k},\zeta^k),
\end{equation}
let $k:=k+1$ and return to Step 1.
\end{algorithm}
\bigskip

\begin{remark}\label{Remark3.1}
\begin{enumerate}[{\rm i)}]
  \item  In Step 1.1 of Algorithm 3.1, by the properties of metric projection onto a nonempty closed convex set, the projection
$$\tilde{x}^{k}=\Pi_{\mathcal{C}}\{x^k-\frac{1}{\beta r}[F(x^{k})-\lambda^k+\beta(x^k-y^k)]\}$$
can be rewritten as
\begin{equation}\label{proj for x inner prud}
\dbE\langle x^k-\frac{1}{\beta r}[F(x^{k})-\lambda^k+\beta(x^k-y^k)]- \tilde{x}^{k},z-\tilde{x}^{k}\rangle\leq 0,\quad \forall z\in \mathcal{C}.
\end{equation}
By Lemma \ref{appendix6},  the inequality (\ref{proj for x inner prud}) is equivalent to the following  pointwise form
$$
\langle x^k(\o)-\frac{1}{\beta r}[F(x^{k})(\o)-\lambda^k(\o)+\beta(x^k(\o)-y^k(\o))]- \tilde{x}^{k}(\o),z-\tilde{x}^{k}(\o)\rangle\leq 0,\forall z\in C(\o),\ \mbox{a.s. }\o\in\O.
$$
Therefore, to compute $\tilde{x}^{k}$ we can compute pointwisely  the projections in finite dimensional space
\begin{equation*}
\tilde{x}^{k}(\o)=\Pi_{C(\o)}\big\{x^k(\o)-\frac{1}{\beta r}[F(x^{k})(\o)-\lambda^k(\o)+\beta(x^k(\o)-y^k(\o))]\big\}
\end{equation*}
for almost every $\o\in \O$. This approach is feasible when the sample space $\O$ is a finite set. In the general cases, solving $\tilde{x}^{k}$ is equivalent to solving the following convex stochastic optimization problem  without the nonanticipativity constraint:
$$\min_{z\in \mathcal{C}} \dbE \Big|z -x^k+\frac{1}{\beta r}[F(x^{k})-\lambda^k+\beta(x^k-y^k)] \Big|^2.$$
  \item In Step 1.2 of Algorithm 3.1, we need to compute the projection onto the closed linear subspace $\mathcal{N}$, which is quite different from the projection in Step 1.1. As illustrated in the above, in Step 1.1., the projection can be transformed into a collection of metric projections (with  parameter $\o\in \O$) in finite dimensional space. On the other hand, to calculate the projection onto the closed linear subspace $\mathcal{N}$ one needs to compute a collection of conditional expectations. Indeed, by Remark   \ref{Rem projection CE}, $\tilde{y}^{k}=(\tilde{y}^{k}_{0},\tilde{y}^{k}_{1},...,\tilde{y}^{k}_{N-1} )$ can be represented by
      $$\tilde{y}^{k}_{i}=\dbE\Big\{y^k_{i}-\frac{1}{\beta}
\big[\lambda^k_{i}-\beta(\tilde{x}^{k}_{i}-y^k_{i})\big]\Big|\ \mathscr{F}_{i}\Big\},\quad i=0,1,...,N-1. $$
If $\xi_{1},\xi_2,...,\xi_{N-1}$ are discrete random vectors,  then, by Example
\ref{example for CE}, the explicit formulas of those conditional expectations can be obtained. In the general cases, the Monte Carlo method has to be used to compute those conditional expectations. That will be much more complicated and time-consuming.
\end{enumerate}
\end{remark}

\begin{remark}
In \cite{HeJCM2006},  two correction approaches are given. Note that in the second correction approach given by \cite{HeJCM2006} one needs to compute one more metric projection. As we have seen in Remark  \ref{Remark3.1}, to compute the metric projection in the stochastic cases will take up much time. Therefore, we adapt here only the first approach of \cite{HeJCM2006}.
\end{remark}

\subsection{Convergence}

In this subsection, we shall prove the weak convergence of the  sequence generated by Algorithm 3.1.

Let $\mathcal{K}:=\mathcal{C}\times\mathcal{N}\times\mathcal{L}^2(\Omega,\mathscr{F}, R^n)$ and define map $T$ from  $\big[\mathcal{L}^2(\Omega,\mathscr{F}, R^n)\big]^3$ to $ \big[\mathcal{L}^2(\Omega,\mathscr{F}, R^n)\big]^3$ by
\begin{equation}\label{thetaT(theta)}
T(\theta):=\begin{pmatrix}
F(x)-\lambda\\
\lambda\\
x-y
\end{pmatrix}, \quad \forall \
\theta=\begin{pmatrix}
 x\\
y\\
\lambda
\end{pmatrix}\in \mathcal{K}.
\end{equation}
Obviously, under condition (A1)-(A2), $T$ is monotone and Lipschitz  continuous.
The following result gives the relationship among  $\vi{T}{\mathcal{K}}$,   $\mbox{MSVI}(F,\mathcal{C}\cap\mathcal{N} )$ (\ref{SVIbasicdefinition}) and the multistage stochastic variational inequality in  extensive form \eqref{SVIextensivedefini}.

\begin{lemma}\label{lemma 3.1}
The following assertions hold true.
\begin{enumerate}[i)]
  \item $\theta^*=(x^*, y^*, \lambda^*)^{\top}$ is a solution to variational inequality \vi{T}{\mathcal{K}} if and only if  $x^*=y^*$  and $(x^*,-\lambda^*)$ is a solution to the multistage stochastic variational inequality in extensive form \eqref{SVIextensivedefini}.
  \item  If $\theta^*=(x^*, y^*, \lambda^*)^{\top}$ is a solution to variational inequality \vi{T}{\mathcal{K}},
then $x^*=y^*$, and $x^*$ is a solution to $\mbox{MSVI}(F,\mathcal{C}\cap\mathcal{N} )$ (\ref{SVIbasicdefinition}).
  \item If the sum rule $N_{\mathcal{C}\cap \mathcal{N}}(x^*)=N_{\mathcal{C}}(x^*)+N_{\mathcal{N}}(x^*)$ holds true and $x^*$ is a solution to $\mbox{MSVI}(F,\mathcal{C}\cap\mathcal{N} )$ (\ref{SVIbasicdefinition}), then there is $\lambda^*\in \mathcal{M}$ such that $\theta^*:=( x^*,x^*,\lambda^*)^{\top}$
is a solution to \vi{T}{\mathcal{K}}.
\end{enumerate}
\end{lemma}

\begin{proof}
By the definition of $T$ and $\mathcal{K}$, $\theta^*$ is a solution to variational inequality \vi{T}{\mathcal{K}} if and only if $x^*=y^*$, $y^*\in\mathcal{N}$, $-\lambda^*\in \mathcal{M}$, and $$-F(x^*)+\lambda^*\in N_{\mathcal{C}}(x^* ).$$
This proves the first assertion. Then, by \cite[Theorem 3.2]{Rockafellar Wets17}, the second and the third assertions hold true.
\end{proof}

From Lemma 3.1 we see that a solution to  $\mbox{MSVI}(F,\mathcal{C}\cap\mathcal{N} )$ (\ref{SVIbasicdefinition}) can be found by solving the variational inequality \vi{T}{\mathcal{K}} if variational inequality \vi{T}{\mathcal{K}} or equivalently, the  multistage stochastic variational inequality in extensive form \eqref{SVIextensivedefini} has a solution. In what follows, we always assume that the solution set of variational inequality \vi{T}{\mathcal{K}} is nonempty.

Let $G$ be defined by (\ref{G}). Define the $G$-norm  and $G$-inner product on $\big[\mathcal{L}^2(\Omega,\mathscr{F}, R^n)\big]^3$  respectively by
\begin{equation}\label{normtheta}
\norm{\theta}_{\mathcal{L}^2,{G}}^2=\beta r\norm{x}_{\mathcal{L}^2}^2+\beta \norm{y}_{\mathcal{L}^2}^2+\frac{1}{\beta} \norm{\lambda}_{\mathcal{L}^2}^2,
\end{equation}
and
\begin{equation}\label{innerproductthetav}
\langle\theta,\vartheta\rangle _{\mathcal{L}^2,{G}}=\langle\theta,G\vartheta\rangle_{\mathcal{L}^2}=
\beta r\langle x,u\rangle_{\mathcal{L}^2}+\beta \langle y,v\rangle_{\mathcal{L}^2}+\frac{1}{\beta}\langle\lambda,w\rangle_{\mathcal{L}^2},
\end{equation}
for any
\begin{equation*}
\theta=\begin{pmatrix}
x\\
y\\
\lambda
\end{pmatrix},\quad
\vartheta=\begin{pmatrix}
 u\\
v\\
w
\end{pmatrix}\in\big[\mathcal{L}^2(\Omega,\mathscr{F}, R^n)\big]^3.
\end{equation*}
In addition, using the $G$-norm, we define the $G$-metric projection operator on to $\mathcal{K}$ by
\begin{equation}\label{PiKGtheta}
\Pi_{\mathcal{K},G}(\theta)=\big\{ \bar\theta\in \mathcal{K}\ \big|\ \norm{\bar\theta-\theta}_{\mathcal{L}^2,G}=\min_{\vartheta\in \mathcal{K}}\norm{\vartheta-\theta}_{\mathcal{L}^2,G}\big\}.
\end{equation}
Clearly, the $G$-metric projection has the following properties:
\begin{enumerate}[{\rm i)}]
  \item $\langle\theta-\Pi_{\mathcal{K},G}(\theta), \vartheta-\Pi_{\mathcal{K},G}(\theta)\rangle_{\mathcal{L}^2,G}\leq0,\ \forall\ \vartheta\in\mathcal{K}.$

\item $\norm{\Pi_{\mathcal{K},G}(\theta)-\Pi_{\mathcal{K},G}(\vartheta)}_{\mathcal{L}^2,G}
    \leq\norm{\theta-\vartheta}_{\mathcal{L}^2,G}$,
    $\forall\ \theta,\vartheta\in\big[\mathcal{L}^2(\Omega,\mathscr{F}, R^n)\big]^3$.
\end{enumerate}

Now, we are in a position to prove the convergence of Algorithm 3.1.
Let $k\in \mathbb{N}$, $\theta^k=(x^k,y^k,\lambda^k)^{\top}\in \mathcal{C}\times \mathcal{N}\times\mathcal{L}^2(\Omega,\mathscr{F}, R^n)$ be the random vector obtained in the $k$-th iteration, and $\tilde \theta^k=(\tilde x^k,\tilde y^k,\tilde \lambda^k)^{\top}\in \mathcal{C}\times \mathcal{N}\times\mathcal{L}^2(\Omega,\mathscr{F}, R^n)$ be the vector generated  by Step 1 of Algorithm 3.1.

Define
\begin{equation}\label{varphithetaktildexik}
\varphi(\theta^{k},\tilde{\theta}^{k},\zeta^k)=\langle\lambda^{k}-\tilde{\lambda}^{k},\tilde{y}^{k}-y^{k}\rangle_{\mathcal{L}^2}
+\langle\theta^{k}-\tilde\theta^{k}, Gd(\theta^k,\tilde{\theta}^{k},\zeta^k)\rangle_{\mathcal{L}^2},
\end{equation}
where $d(\theta^k,\tilde{\theta}^{k},\zeta^k)$ is defined by (\ref{theaktildexik}). We have the following estimate for $\varphi(\theta^{k},\tilde{\theta}^{k},\zeta^k)$.

\begin{proposition}\label{propositionvaphid}
Assume (A2). Choose $r>1+\frac{L_{F}}{\beta}$, where $L_{F}$ is the $Lipschitz$ constant of $F$. Then
\begin{equation}
\varphi(\theta^{k},\tilde{\theta}^{k},\zeta^k)\ge\frac{1}{2}\norm{d(\theta^{k},\tilde{\theta}^{k},\zeta^k)}_{\mathcal{L}^2,G}^{2}.
\end{equation}
\end{proposition}
\begin{proof}
By the definition of $d(\theta^k,\tilde{\theta}^{k},\zeta^k)$,
\begin{eqnarray}\label{proposition1}
\varphi(\theta^{k},\tilde{\theta}^{k},\zeta^k)&=&\langle\lambda^{k}-\tilde{\lambda}^{k},\tilde{y}^{k}-y^{k}\rangle_{\mathcal{L}^2}
+\langle x^k-\tilde{x}^k,r\beta(x^k-\tilde{x}^k)-\zeta_{x}^k\rangle_{\mathcal{L}^2}\nonumber\\
&&+\langle y^k-\tilde{y}^k,\beta(y^k-\tilde{y}^k)\rangle_{\mathcal{L}^2}
+\langle \lambda^k-\tilde{\lambda}^k,\frac{1}{\beta}(\lambda^k-\tilde{\lambda}^k)\rangle_{\mathcal{L}^2}.
\end{eqnarray}

By $\lambda^k-\tilde{\lambda}^k=\beta(\tilde{x}^k-\tilde{y}^k)$, we have
\begin{eqnarray}\label{proposition2}
&&\langle\lambda^{k}-\tilde{\lambda}^k,\tilde{y}^k-y^k\rangle_{\mathcal{L}^2}+\frac{1}{2}\langle \lambda^k-\tilde{\lambda}^k,\frac{1}{\beta} (\lambda^k-\tilde{\lambda}^k)\rangle_{\mathcal{L}^2}+
\frac{1}{2}\langle y^k-\tilde{y}^k,\beta(y^k-\tilde{y}^k)\rangle_{\mathcal{L}^2}\nonumber\\
&=&\frac{\beta}{2}\norm{y^k-\tilde{y}^k-\frac{1}{\beta}(\lambda^k-\tilde{\lambda}^k)}_{\mathcal{L}^2}^2\nonumber\\
&=&\frac{\beta}{2}\norm{y^k-\tilde{y}^k-\tilde{x}^k+\tilde{y}^k}_{\mathcal{L}^2}^2\nonumber\\
&=&\frac{\beta}{2}\norm{y^k-\tilde{x}^k}_{\mathcal{L}^2}^2.
\end{eqnarray}
Let $\zeta_{x}^k$ be defined by (\ref{SVIxixk}). Then, by the Lipschitz continuity of $F$,
$$
\norm{\zeta_{x}^k}_{\mathcal{L}^2}=\norm{\beta(x^k-\tilde{x}^k)+F(x^k)-F(\tilde{x}^k)}_{\mathcal{L}^2}
\leq (L_{F}+\beta)\norm{x^k-\tilde{x}^k}_{\mathcal{L}^2}
\le r\beta\norm{x^k-\tilde{x}^k}_{\mathcal{L}^2}.
$$
It implies that
\begin{eqnarray}\label{proposition5}
&&\frac{1}{2}\langle \lambda^{k}-\tilde{\lambda}^k,\frac{1}{\beta}(\lambda^k-\tilde{\lambda}^k)\rangle_{\mathcal{L}^2}
+\frac{1}{2} \langle y^k-\tilde{y}^k,\beta(y^k-\tilde{y}^k)\rangle_{\mathcal{L}^2}
+\langle x^k-\tilde{x}^k,r\beta(x^k-\tilde{x}^k)-\zeta_{x}^k\rangle_{\mathcal{L}^2}\nonumber\\
&\ge&\frac{1}{2} \langle \lambda^{k}-\tilde{\lambda}^k,\frac{1}{\beta}(\lambda^k-\tilde{\lambda}^k)\rangle_{\mathcal{L}^2}
+\frac{1}{2} \langle y^k-\tilde{y}^k,\beta(y^k-\tilde{y}^k)\rangle_{\mathcal{L}^2}\nonumber\\
&&\quad+\frac{1}{2}\langle x^k-\tilde{x}^k,r\beta(x^k-\tilde{x}^k)\rangle_{\mathcal{L}^2}\nonumber
-\langle x^k-\tilde{x}^k,\zeta_{x}^k\rangle_{\mathcal{L}^2}+\frac{1}{2r\beta} \langle\zeta_{x}^{k},\zeta_{x}^k \rangle_{\mathcal{L}^2}\nonumber\\
&=&\frac{1}{2\beta}\norm{\lambda^k-\tilde{\lambda}^{k}}_{\mathcal{L}^2}^2
+\frac{\beta}{2}\norm{y^k-\tilde{y}^k}_{\mathcal{L}^2}^{2}
+\frac{1}{2}\langle x^k-\tilde{x}^k
-\frac{1}{r\beta}\zeta_{x}^k,r\beta(x^k-\tilde{x}^k
-\frac{1}{r\beta}\zeta_{x}^k)\rangle_{\mathcal{L}^2}\nonumber\\
&=&\frac{1}{2}\norm{d(\theta^k,\tilde{\theta}^k,\zeta^k)}_{\mathcal{L}^2,G}^{2}   .
\end{eqnarray}
Combining \eqref{proposition1}, \eqref{proposition2} with \eqref{proposition5}, we obtain that
$$
\varphi(\theta^k,\tilde{\theta}^k,\zeta^k)
\ge\frac{\beta}{2}\norm{y^k-\tilde{x}^k}_{\mathcal{L}^2}^{2}
+\frac{1}{2}\norm{d(\theta^k,\tilde{\theta}^k,\zeta^k)}_{\mathcal{L}^2,G}^{2}
\ge\frac{1}{2}\norm{d(\theta^k,\tilde{\theta}^k,\zeta^k)}_{\mathcal{L}^2,G}^{2}.
$$
\end{proof}

Denote by  $Sol~ \vi{T}{\mathcal{K}}$  the solution set of \vi{T}{\mathcal{K}}.
The following theorem shows  the contractive property of the sequence $\{\theta^{k}\}$ generated by Algorithm 3.1.

\begin{theorem}\label{theorem3.1}
Assume (A1)-(A2) and let $r > \frac{L_{F}}{\beta}+1$, $\alpha\in (0,1)$.  Then, for any $\theta^* \in Sol~ \vi{T}{\mathcal{K}}$,
\begin{equation}\label{theorem1:0}
\norm{\theta^{k+1}-\theta^*}_{\mathcal{L}^2,G}^2\leq\norm{\theta^k-\theta^*}_{\mathcal{L}^2,G}^2
-\alpha(1-\alpha)\norm{d(\theta^k,\tilde{\theta}^k,\zeta^k)}_{\mathcal{L}^2,G}^2.
\end{equation}
\end{theorem}

\begin{proof}
By the definition of $\theta^{k+1}$, we have
\begin{eqnarray}\label{theorem1:1}
\norm{\theta^{k+1}-\theta^*}_{\mathcal{L}^2,G}^2
&=&\norm{\theta^k-\alpha d(\theta^k,\tilde{\theta}^k,\zeta^k)-\theta^*}_{\mathcal{L}^2,G}^2\nonumber\\[+0.4em]
&=&\norm{\theta^{k}-\theta^*}_{\mathcal{L}^2,G}^2
+\alpha^2\norm{d(\theta^k,\tilde{\theta}^k,\zeta^k)}_{\mathcal{L}^2,G}^2
-2\alpha \langle\theta^k-\theta^*,Gd(\theta^k,\tilde{\theta}^k,\zeta^k)\rangle_{\mathcal{L}^2}\nonumber\\[+0.4em]
&=&\norm{\theta^{k}-\theta^*}_{\mathcal{L}^2,G}^2
+\alpha^2\norm{d(\theta^k,\tilde{\theta}^k,\zeta^k)}_{\mathcal{L}^2,G}^2\nonumber\\[+0.4em]
&&-2\alpha \langle \theta^k-\tilde{\theta}^k,G d(\theta^k,\tilde{\theta}^k,\zeta^k)\rangle_{\mathcal{L}^2}+2\alpha \langle \theta^*-\tilde{\theta}^k,G d(\theta^k,\tilde{\theta}^k,\zeta^k)\rangle_{\mathcal{L}^2}\nonumber\\[+0.4em]
&&-2\alpha \langle \lambda^k-\tilde{\lambda}^k,\tilde{y}^k-y^k\rangle_{\mathcal{L}^2}+2\alpha \langle \lambda^k-\tilde{\lambda}^k,\tilde{y}^k
-y^k\rangle_{\mathcal{L}^2}\nonumber\\[+0.4em]
&=&\norm{\theta^k-\theta^*}_{\mathcal{L}^2,G}^2+\alpha^2\norm{d(\theta^k,\tilde{\theta}^k,\zeta^k)}_{\mathcal{L}^2,G}^2
-2\alpha\varphi(\theta^k,\tilde{\theta}^k,\zeta^k)\nonumber\\[+0.4em]
&&+2\alpha \langle \theta^*-\tilde{\theta}^k,G d(\theta^k,\tilde{\theta}^k,\zeta^k)\rangle_{\mathcal{L}^2}
+2\alpha \langle \lambda^k-\tilde{\lambda}^k,\tilde{y}^k-y^k\rangle_{\mathcal{L}^2}.
\end{eqnarray}

We claim that
\begin{equation}\label{theorem1:2}
\langle \theta^*-\tilde{\theta}^k,G d(\theta^k,\tilde{\theta}^k,\zeta^k)\rangle_{\mathcal{L}^2}+\langle \lambda^k-\tilde{\lambda}^k,\tilde{y}^k-y^k\rangle_{\mathcal{L}^2}\leq0.
\end{equation}
Indeed, since $\theta^*\in Sol~ \vi{T}{\mathcal{K}}$, we have $x^*\in \mathcal{C}$, $y^*\in\mathcal{N}$, $x^*=y^*$, $\lambda^*\in\mathcal{M}$, and
\begin{equation}\label{theorem1:3}
\langle  F(x^*)-\lambda^*, x-x^*\rangle_{\mathcal{L}^2}\geq0,~\forall\ x\in \mathcal{C}.
\end{equation}
Note that $\lambda^k-\tilde{\lambda}^k=\beta(\tilde{x}^k-\tilde{y}^k)$. By (\ref{SVIxixk}) and the definition of $\tilde{x}^k$,
\begin{eqnarray*}
\tilde{x}^{k}&=&\Pi_{\mathcal{C}}\big\{x^k-\frac{1}{\beta r}
[F(x^{k})-\lambda^k+\beta(x^k-y^k)]\big\}\nonumber\\
&=&\Pi_{\mathcal{C}}\big\{x^k-\frac{1}{\beta r}[F(\tilde{x}^k)-\tilde{\lambda}^k
-\beta(y^k-\tilde {y}^{k})+\zeta_{x}^k]\big\}.
\end{eqnarray*}
Then, we have
\begin{equation}\label{theorem1:5}
\langle x^*-\tilde{x}^k,F(\tilde{x}^k)-\tilde{\lambda}^k-\beta(y^k-\tilde{y}^k)+\zeta_{x}^k-\beta r(x^k-\tilde{x}^k)\rangle_{\mathcal{L}^2}\geq0.
\end{equation}
Combining \eqref{theorem1:3} with \eqref{theorem1:5}, we obtain from the monotonicity of $F$ that
\begin{equation}\label{theorem1:7}
\langle x^*-\tilde{x}^k,\lambda^*-\tilde{\lambda}^k-\beta(y^k-\tilde{y}^k)+\zeta_{x}^k-\beta r(x^k-\tilde{x}^k)\rangle_{\mathcal{L}^2}\geq 0.
\end{equation}

Similarly, by
$$\tilde{y}^k=\Pi_{\mathcal{N}}\big\{y^k-\frac{1}{\beta}[\tilde{\lambda}^k-\beta(\tilde{y}^k-y^k)]\big\},$$
we have $\tilde{y}^k\in \mathcal{N}$ and
\begin{equation}\label{theorem1:6}
\langle y^*-\tilde{y}^k,\tilde{\lambda}^k-\beta(\tilde{y}^k-y^k)-\beta(y^k-\tilde{y}^k)\rangle_{\mathcal{L}^2}=\langle y^*-\tilde{y}^k,\tilde{\lambda}^k\rangle_{\mathcal{L}^2}\geq0.
\end{equation}

Since $\lambda^*\in \mathcal{M}$, $y^*,\tilde{y}^k\in \mathcal{N}$, $\inner{\lambda^*}{y^*-\tilde{y}^k}_{\mathcal{L}^2}=0$. Then,  by \eqref{theorem1:6},
$$\langle y^*-\tilde{y}^k,\tilde{\lambda}^k-\lambda^*\rangle_{\mathcal{L}^2}\geq0.
$$
It implies that
\begin{align}\label{theorem1:8}
&\langle y^*-\tilde{y}^k,\beta(y^k-\tilde{y}^k)+(\tilde{\lambda}^k-\lambda^*)\rangle_{\mathcal{L}^2}
\geq \langle y^*-\tilde{y}^k,\beta(y^k-\tilde{y}^k)\rangle_{\mathcal{L}^2}.
\end{align}
In addition,  by $\tilde{\lambda}^k=\lambda^k-\beta(\tilde{x}^k-\tilde{y}^k)$ and $x^*-y^*=0,$
\begin{equation}\label{theorem1:9}
\langle\lambda^*-\tilde{\lambda}^k, \frac{1}{\beta}(\tilde{\lambda}^k-\lambda^k)\rangle_{\mathcal{L}^2}
=\langle \lambda^*-\tilde{\lambda}^k,\tilde{y}^k-\tilde{x}^k\rangle_{\mathcal{L}^2}
=\langle\lambda^*-\tilde{\lambda}^k,\tilde{y}^k-y^*-(\tilde{x}^k-x^*)\rangle_{\mathcal{L}^2}.
\end{equation}
Combining \eqref{theorem1:7}, \eqref{theorem1:8} with \eqref{theorem1:9}, we  obtain
\begin{eqnarray}\label{theorem1:10}
&&\langle\theta^*-\tilde{\theta}^k,Gd(\theta^k,\tilde{\theta}^k,\zeta^k)\rangle_{\mathcal{L}^2}\nonumber\\
&=&\langle x^*-\tilde{x}^k,r \beta(x^k-\tilde{x}^k)-\zeta_{x}^k\rangle_{\mathcal{L}^2}
+\langle y^*-\tilde{y}^k,\beta(y^k-\tilde{y}^k)\rangle_{\mathcal{L}^2}\nonumber\\
&&+\langle \lambda^*-\tilde{\lambda}^k,\frac{1}{\beta}(\lambda^k-\tilde{\lambda}^k)\rangle_{\mathcal{L}^2}\nonumber\\
&\leq& \langle x^*-\tilde{x}^k,\lambda^*-\tilde{\lambda}^k-\beta(y^k-\tilde{y}^k)\rangle_{\mathcal{L}^2}
+\langle y^*-\tilde{y}^k,\tilde{\lambda}^k-\lambda^*-\beta(\tilde{y}^k-y^k)\rangle_{\mathcal{L}^2}\nonumber\\
&&+\langle\lambda^*-\tilde{\lambda}^k, (\tilde{x}^k-x^*)-(\tilde{y}^k-y^*)\rangle_{\mathcal{L}^2}\nonumber\\[+0.4em]
&=&\langle x^*-\tilde{x}^k-y^*+\tilde{y}^k,\beta(\tilde{y}^k-y^k)\rangle_{\mathcal{L}^2}\nonumber\\[+0.4em]
&=&\langle\tilde{y}^k-\tilde{x}^k,\beta(\tilde{y}^k-y^k)\rangle_{\mathcal{L}^2}\nonumber\\[+0.4em]
&=&\langle\tilde{\lambda}^k-\lambda^k,\tilde{y}^k-y^k\rangle_{\mathcal{L}^2}.
\end{eqnarray}
This proves (\ref{theorem1:2}).

Finally,  by \eqref{theorem1:1}, \eqref{theorem1:2} and Proposition \ref{propositionvaphid}, we have
\begin{eqnarray*}
&&\norm{\theta^{k+1}-\theta^*}_{\mathcal{L}^2,G}^2\nonumber\\
&\leq&\norm{\theta^k-\theta^*}_{\mathcal{L}^2,G}^2
+\alpha^2\norm{d(\theta^k,\tilde{\theta}^k,\zeta^k)}_{\mathcal{L}^2,G}^2
-2\alpha \varphi(\theta^k,\tilde{\theta}^k,\zeta^k)\nonumber\\
&\le&\norm{\theta^k-\theta^*}_{\mathcal{L}^2,G}^2+\alpha^2\norm{d(\theta^k,\tilde{\theta}^k,\zeta^k)}_{\mathcal{L}^2,G}^2
-\alpha||d(\theta^k,\tilde{\theta}^k,\zeta^k)||_{\mathcal{L}^2,G}^2\nonumber\\
&=&\norm{\theta^k-\theta^*}_{\mathcal{L}^2,G}^2-\alpha(1-\alpha)\norm{d(\theta^k,\tilde{\theta}^k,\zeta^k)}_{\mathcal{L}^2,G}^2.
\end{eqnarray*}
This completes the proof of Theorem \ref{theorem3.1}
\end{proof}

\begin{theorem}\label{theorem2}
Assume (A1)-(A2). Let $r > \frac{L_{F}}{\beta}+1$ and $\alpha\in (0,1)$.
Then, the sequence $\{\theta^k\}_{k=0}^{\infty}$ generated by Algorithm 3.1 converges weakly to some $\theta^*\in Sol~\vi{T}{\mathcal{K}}$.
\end{theorem}
\begin{proof}
The proof is divided into four steps.

Step 1: In this step, we shall prove that there is a constant $K>0$ s.t.
\begin{equation}\label{theorem2:2}
\norm{e_{G}(\tilde{\theta}^k,T,\mathcal{K})}_{\mathcal{L}^2,G}\leq K\norm{\theta^k-\tilde{\theta}^k}_{\mathcal{L}^2,G},
\end{equation}
where
\begin{equation}\label{theorem2:1}
e_{G}(\tilde{\theta}^k,T,\mathcal{K})=\tilde{\theta}^k-
\Pi_{\mathcal{K},G}(\tilde{\theta}^k-G^{-1}T(\tilde{\theta}^k)).
\end{equation}

Define $\Lambda:\mathcal{L}^2(\Omega,\mathscr{F},R^n)\rightarrow \big[\mathcal{L}^2(\Omega,\mathscr{F}, R^n)\big]^{3}$ by
\begin{equation}\label{Lambda}
\Lambda(x)=\begin{pmatrix}
x\\
-x\\
0
\end{pmatrix}, \quad \forall \ x\in \mathcal{L}^2(\Omega,\mathscr{F}, R^n).
\end{equation}
Then, the Step 1 of Algorithm 3.1 can be rewritten as
\begin{equation}\label{theorem1:13}
\tilde{\theta}^{k}=\Pi_{\mathcal{K},G}\big\{\theta^{k}
-G^{-1}[T(\tilde{\theta}^{k})-\beta\Lambda(y^{k}-\tilde{y}^{k})+\zeta^k]\big\}.
\end{equation}

It follows from the Lipschitz continuity of $G$-metric projection  that
\begin{eqnarray}\label{theorem2:3+}
&&\norm{e_{G}(\tilde{\theta}^k,T,\mathcal{K})}_{\mathcal{L}^2,G}\nonumber\\[+0.4em]
&=&\norm{\Pi_{\mathcal{K},G}\big\{\theta^k-G^{-1}[T(\tilde{\theta}^k)
-\beta\Lambda(y^k-\tilde{y}^k)+\zeta^k]\big\}-\Pi_{\mathcal{K},G}(\tilde{\theta}^k
-G^{-1}T(\tilde{\theta}^k))}_{\mathcal{L}^2,G}\nonumber\\[+0.4em]
&\leq & \norm{\theta^{k}-\tilde{\theta}^k-G^{-1}[\beta\Lambda(\tilde{y}^k-y^k)+\zeta^k]}_{\mathcal{L}^2,G}\nonumber\\[+0.4em]
&\leq& \norm{\theta^k-\tilde{\theta}^k}_{\mathcal{L}^2,G}+ \norm{G^{-1}\beta\Lambda(\tilde{y}^k-y^k) }_{\mathcal{L}^2,G}
+\norm{G^{-1}\zeta^k}_{\mathcal{L}^2,G}.
\end{eqnarray}

By the definition of $\Lambda$,
\begin{eqnarray*}
\norm{G^{-1}\beta\Lambda(\tilde{y}^k-y^k) }_{\mathcal{L}^2,G}
&=&\Bigg\|G^{-1}\begin{pmatrix}
\beta(\tilde{y}^k-y^k)\\
-\beta(\tilde{y}^k-y^k)\\
0
\end{pmatrix}\Bigg\|_{\mathcal{L}^2,G}\\
&=&\Bigg\|\begin{pmatrix}
\frac{1}{r}(\tilde{y}^k-y^k)\\
-(\tilde{y}^k-y^k)\\
0
\end{pmatrix}\Bigg\|_{\mathcal{L}^2,G}\\
&=&\sqrt{\frac{1+r}{r}}\Big[\beta \|\tilde{y}^k-y^k\|_{\mathcal{L}^2}^{2}\Big]^{\frac{1}{2}}\\
&\le&\sqrt{\frac{1+r}{r}} \|\tilde{\theta}^k-\theta^k\|_{\mathcal{L}^2,G}.
\end{eqnarray*}
It follows from \eqref{theorem2:3+} that
\begin{eqnarray}\label{theorem2:3}
&&\norm{e_{G}(\tilde{\theta}^k,T,\mathcal{K})}_{\mathcal{L}^2,G}\nonumber\\[+0.4em]
&\leq& \norm{\theta^k-\tilde{\theta}^k}_{\mathcal{L}^2,G}+
\sqrt{\frac{1+r}{r}} \|\tilde{\theta}^k-\theta^k\|_{\mathcal{L}^2,G}
+\norm{G^{-1}\zeta^k}_{\mathcal{L}^2,G}.
\end{eqnarray}

In addition, by the definition of $\zeta^k$,
\begin{eqnarray}\label{theorem2:5}
\norm{G^{-1}\zeta^k}_{\mathcal{L}^2,G}^2
&=&\frac{1}{r\beta}\norm{\zeta_{x}^k}_{\mathcal{L}^2}^2\nonumber\\
&=&\frac{1}{r\beta}\norm{F(x^k)-F(\tilde{x}^k)+\beta(x^k-\tilde{x}^k)}_{\mathcal{L}^2}^2\nonumber\\
&\leq& \frac{1}{r\beta}(L_{F}+\beta)^2\norm{x^k-\tilde{x}^k}_{\mathcal{L}^2}^2\nonumber\\
&\leq&(\frac{L_{F}+\beta}{r\beta})^2\norm{\theta^k-\tilde{\theta}^k}_{\mathcal{L}^2,G}^2.
\end{eqnarray}

Combining (\ref{theorem2:3}) with (\ref{theorem2:5}), we obtain that
\begin{equation*}
\norm{e_{G}(\tilde{\theta}^k,T,\mathcal{K})}_{\mathcal{L}^2,G}\leq K\norm{\theta^k-\tilde{\theta}^k}_{\mathcal{L}^2,G},
\end{equation*}
where $$K=1+\sqrt{\frac{1+r}{r}} +\frac{L_{F}+\beta}{r\beta}.$$

Step 2: In this step, we shall prove that, for any $\theta^*\in Sol~\vi{T}{\mathcal{K}}$,
\begin{equation}\label{theorem2:7}
\norm{\theta^{k+1}-\theta^*}_{\mathcal{L}^2,G}^2\leq\norm{\theta^k-\theta^*}_{\mathcal{L}^2,G}^2
-(\alpha-\alpha^2)(1-\frac{L_{F}+\beta}{r\beta})^2\norm{\theta^k-\tilde{\theta}^k}_{\mathcal{L}^2,G}^2.
\end{equation}
By Theorem \ref{theorem3.1},
\begin{equation}\label{theorem2:8}
\norm{\theta^{k+1}-\theta^*}_{\mathcal{L}^2,G}^2\leq\norm{\theta^k-\theta^*}_{\mathcal{L}^2,G}^2
-(\alpha-\alpha^2)\norm{d(\theta^k,\tilde{\theta}^k,\zeta^k)}_{\mathcal{L}^2,G}^2.
\end{equation}
Since $r>\frac{L_F}{\beta}+1$, $\frac{L_{F}+\beta}{r\beta}\in(0,1)$. Then, by \eqref{theorem2:5}, we have
\begin{align}\label{theorem2:9}
\norm{d(\theta^k,\tilde{\theta}^k,\zeta^k)}_{\mathcal{L}^2,G}^2&=\norm{\theta^k-\tilde{\theta}^k-G^{-1}\zeta^k}_{\mathcal{L}^2,G}^2\nonumber\\
&=\norm{\theta^k-\tilde{\theta}^k}_{\mathcal{L}^2,G}^2-2\langle\theta^k-\tilde{\theta}^k,G^{-1}\zeta^k\rangle_{\mathcal{L}^2,G}
+\norm{G^{-1}\zeta^k}_{\mathcal{L}^2,G}^2\nonumber\\
&\geq\norm{\theta^k-\tilde{\theta}^k}_{\mathcal{L}^2,G}^2
-2\norm{\theta^k-\tilde{\theta}^k}_{\mathcal{L}^2,G}\norm{G^{-1}\zeta^k}_{\mathcal{L}^2,G}
+\norm{G^{-1}\zeta^k}_{\mathcal{L}^2,G}^2 \nonumber\\
&=\big(\norm{\theta^k-\tilde{\theta}^k}_{\mathcal{L}^2,G}
-\norm{G^{-1}\zeta^k}_{\mathcal{L}^2,G}\big)^2\nonumber\\
&\geq(1-\frac{L_{F}+\beta}{r\beta})^2\norm{\theta^k-\tilde{\theta}^k}_{\mathcal{L}^2,G}^2.
\end{align}
Substituting (\ref{theorem2:9}) into (\ref{theorem2:8}), we obtain (\ref{theorem2:7}).

Step 3: In this step, we prove that any weak cluster point of $\{\theta^k\}_{k=0}^{\infty}$ is a solution of \vi{T}{\mathcal{K}}.

By (\ref{theorem2:7}), $\{\theta^k\}_{k=0}^{\infty}$ is bounded. Moreover,
\begin{align*}
(\alpha-\alpha^2)(1-\frac{L_{F}+\beta}{r\beta})^2\sum\limits_{k=0}^{m}\norm{\theta^k-\tilde{\theta}^k}_{\mathcal{L}^2,G}^2
&\leq\norm{\theta^0-\theta^*}_{\mathcal{L}^2,G}^2-\norm{\theta^{m+1}-\theta^*}_{\mathcal{L}^2,G}^2\nonumber\\
&\leq\norm{\theta^0-\theta^*}_{\mathcal{L}^2,G}^2, \quad \forall\ m\in\mathbb{N}.
\end{align*}
Letting $m\to \infty$, we obtain that
\begin{equation}\label{theorem2:12}
(\alpha-\alpha^2)(1-\frac{L_{F}+\beta}{r\beta})^2\sum\limits_{k=0}^{\infty}\norm{\theta^k-\tilde{\theta}^k}_{\mathcal{L}^2,G}^2
\leq\norm{\theta^0-\theta^*}_{\mathcal{L}^2,G}^2.
\end{equation}
Therefore,
\begin{equation}\label{theorem2:13}
\lim\limits_{k\rightarrow\infty}\norm{\theta^k-\tilde{\theta}^k}_{\mathcal{L}^2,G}^2=0.
\end{equation}

Define
\begin{equation*}
\bar{\theta}^k=\Pi_{\mathcal{K},G}(\tilde{\theta}^k-G^{-1}T(\tilde{\theta}^k)).
\end{equation*}
Clearly,
\begin{align*}
\bar{\theta}^k&=\Pi_{\mathcal{K},G}\big\{\bar{\theta}^k
-G^{-1}[T(\tilde{\theta}^k)+G(\bar{\theta}^k-\tilde{\theta}^k)]\big\}\nonumber\\
&=\Pi_{\mathcal{K}}\big\{\bar{\theta}^k-[T(\tilde{\theta}^k)+G(\bar{\theta}^k-\tilde{\theta}^k)]\big\}.
\end{align*}
Letting $N_{\mathcal{K}}(\bar{\theta}^k)$ be the normal cone to $\mathcal{K}$ at $\bar{\theta}^k\in\mathcal{K}$, it follows that
\begin{equation}\label{theorem2:16}
-T(\tilde{\theta}^k)-G(\bar{\theta}^k-\tilde{\theta}^k)\in N_{\mathcal{K}}(\bar{\theta}^k).
\end{equation}
Adding $T(\bar{\theta}^k)$ to both sides of (\ref{theorem2:16}), we have that
\begin{equation}\label{theorem2:17}
T(\bar{\theta}^k)-T(\tilde{\theta}^k)-G(\bar{\theta}^k-\tilde{\theta}^k)\in N_{\mathcal{K}}(\bar{\theta}^k)+T(\bar{\theta}^k),
\end{equation}
i.e.,
\begin{equation}\label{theorem2:18}
(\bar{\theta}^k,T(\bar{\theta}^k)-T(\tilde{\theta}^k)
-G(\bar{\theta}^k-\tilde{\theta}^k))\in Gph (N_{\mathcal{K}}(\cdot)+T(\cdot)).
\end{equation}
By \eqref{theorem2:2} and \eqref{theorem2:13},
\begin{equation}\label{theorem2:20}
\lim\limits_{k\rightarrow\infty}\norm{\tilde{\theta}^k-\bar{\theta}^k}_{\mathcal{L}^2,G}
=\lim\limits_{k\rightarrow\infty}\norm{e_{G}(\tilde{\theta}^k,T,\mathcal{K})}_{\mathcal{L}^2,G} \leq \lim\limits_{k\rightarrow\infty} K\norm{\theta^k-\tilde{\theta}^k}_{\mathcal{L}^2,G}=0.
\end{equation}
By condition (A2), $T$ is Lipschitz continuous. Then we have
$$\lim\limits_{k\rightarrow \infty}\norm{T(\bar{\theta}^{k})-T(\tilde{\theta}^{k})
-G(\bar{\theta}^{k}-\tilde{\theta}^{k})}_{\mathcal{L}^2,G}=0.$$

Since $\{\theta^k\}_{k=0}^{\infty}$ is bounded, it possesses a weakly convergent subsequence $\{\theta^{k_{j}}\}_{j=0}^{\infty}$. Assume $\theta^{k_{j}}\rightharpoonup\theta^\infty.$ By \eqref{theorem2:13} and  \eqref{theorem2:20}, we have $\tilde{\theta}^{k_{j}}\rightharpoonup\theta^\infty$ and $\bar{\theta}^{k_{j}}\rightharpoonup\theta^\infty$.
Under condition (A1)-(A2), $T$ is Lipschitz continuous and monotone. By Lemma \ref{maxnormalmaximonotone}, $N_{\mathcal{K}}(\cdot)+T(\cdot)$ is maximal monotone.
Then, by \eqref{theorem2:17} and Lemma \ref{maximonotoweakstrongly}, we have
\begin{equation}\label{theorem2:23}
0\in T(\theta^\infty)+ N_{\mathcal{K}}(\theta^\infty),
\end{equation}
i.e., $\theta^{\infty}\in Sol~\vi{T}{\mathcal{K}}$.

Step 4: In this step, we prove that $\{\theta^k\}$ has only one weak cluster point.

Assume $\{\theta^{k_{i}}\}_{i=0}^{\infty}\subset \{\theta^k\}_{k=0}^{\infty}$, $\{\theta^{k_{j}}\}_{j=0}^{\infty}\subset \{\theta^k\}_{k=0}^{\infty}$, $\theta^{k_{i}}\rightharpoonup\theta_{1}^{\infty}$ as $i\to \infty$, and $\theta^{k_{j}}\rightharpoonup\theta_{2}^{\infty}$ as $j\to \infty$. By step 3, $\theta_{1}^{\infty},\theta_{2}^{\infty}\in Sol$\vi{T}{\mathcal{K}}.
Then, replacing $\theta^*$ in \eqref{theorem2:7} by $\theta_{1}^{\infty}$ and $\theta_{2}^{\infty}$  respectively, we have
$\{\norm{\theta^k-\theta_{1}^{\infty}}_{\mathcal{L}^2, G}^2\}$ and $\{\norm{\theta^k-\theta_{2}^{\infty}}_{\mathcal{L}^2, G}^2\}$ are decreasing and bounded.
Assume $\norm{\theta^{k}-\theta_{1}^{\infty}}_{\mathcal{L}^2,G}^2\rightarrow l_{1}$  and$\norm{\theta^{k}-\theta_{2}^{\infty}}_{\mathcal{L}^2,G}^2\rightarrow l_{2}$. Then,
\begin{align}\label{theorem2:25}
&\quad~\norm{\theta^k-\theta_{1}^{\infty}}_{\mathcal{L}^2,G}^2-\norm{\theta^k-\theta_{2}^{\infty}}_{\mathcal{L}^2,G}^2\nonumber\\
&=\norm{\theta^k-\theta_{2}^{\infty}+\theta_{2}^{\infty}-\theta_{1}^{\infty}}_{\mathcal{L}^2,G}^2
-\norm{\theta^k-\theta_{2}^{\infty}}_{\mathcal{L}^2,G}^2\nonumber\\
&=\norm{\theta_{2}^{\infty}-\theta_{1}^{\infty}}_{\mathcal{L}^2,G}^2
+2\langle\theta^{k}-\theta_{2}^{\infty}, \theta_{2}^{\infty}-\theta_{1}^{\infty}\rangle_{\mathcal{L}^2,G}.
\end{align}

By $\theta^{k_{i}}\rightharpoonup\theta_{1}^{\infty}$,
\begin{align}\label{theorem2:26}
l_{1}-l_{2}&=\lim\limits_{i\rightarrow\infty}\big[\norm{\theta^{k_{i}}-\theta_{1}^{\infty}}_{\mathcal{L}^2,G}^2
-\norm{\theta^{k_{i}}-\theta_{2}^{\infty}}_{\mathcal{L}^2,G}^2\big]\nonumber\\
&=\lim\limits_{i\rightarrow\infty}\big[\norm{\theta_{2}^{\infty}-\theta_{1}^{\infty}}_{\mathcal{L}^2,G}^2
+2\langle\theta^{k_{i}}-\theta_{2}^{\infty}, \theta_{2}^{\infty}-\theta_{1}^{\infty}\rangle_{\mathcal{L}^2,G}\big]\nonumber\\
&=\norm{\theta_{2}^{\infty}-\theta_{1}^{\infty}}_{\mathcal{L}^2,G}^2
-2\norm{\theta_{2}^{\infty}-\theta_{1}^{\infty}}_{\mathcal{L}^2,G}^2
\nonumber\\
&=-\norm{\theta_{2}^{\infty}-\theta_{1}^{\infty}}_{\mathcal{L}^2,G}^2.
\end{align}

On the other hand, by $\theta^{k_{j}}\rightharpoonup\theta_{2}^{\infty}$,
\begin{align}\label{theorem2:27}
l_{1}-l_{2}&=\lim\limits_{j\rightarrow\infty}\big[\norm{\theta^{k_{j}}-\theta_{1}^{\infty}}_{\mathcal{L}^2,G}^2
-\norm{\theta^{k_{j}}-\theta_{2}^{\infty}}_{\mathcal{L}^2,G}^2\big]\nonumber\\
&=\lim\limits_{j\rightarrow\infty}\big[\norm{\theta_{2}^{\infty}-\theta_{1}^{\infty}}_{\mathcal{L}^2,G}^2
+2\langle\theta^{k_{j}}-\theta_{2}^{\infty}, \theta_{2}^{\infty}-\theta_{1}^{\infty}\rangle_{\mathcal{L}^2,G}\big]\nonumber\\
&=\norm{\theta_{2}^{\infty}-\theta_{1}^{\infty}}_{\mathcal{L}^2,G}^2.
\end{align}
By \eqref{theorem2:26} and \eqref{theorem2:27}, $\norm{\theta_{2}^{\infty}-\theta_{1}^{\infty}}_{\mathcal{L}^2,G}^2=0$. This proves the uniqueness of the weak cluster.

This completes the proof of Theorem \ref{theorem2}.
\end{proof}

\begin{theorem}\label{theorem3}
Under the conditions in Theorem \ref{theorem2}, the sequence $\{x^{k}\}_{k=0}^{\infty}$ converges weakly to a solution $x^*$ to $\mbox{MSVI}(F,\mathcal{C}\cap\mathcal{N} )$ (\ref{SVIbasicdefinition}).
\end{theorem}
\begin{proof}
It follows directly from Theorem \ref{theorem2} and Lemma \ref{lemma 3.1}.
\end{proof}

\section{The discrete cases and numerical examples}\label{sec4}

In this section, we shall consider the special case that the sample space is a finite set.  On such discrete sample space, all the random vectors are discrete type random vectors, and in such case, as discussed  in \cite{Rockafellar Wets17}, the multistage stochastic variational inequality is actually defined on a finite dimensional Hilbert space.

Let $m\in \mathbb{N}$. Consider the sample space
$$\hat\O:=\big\{\hat\o_1,\hat\o_2,...,\hat\o_m\big\}.$$
Define $\hat{\mathscr{F}}=2^{\hat\O}$, i.e., $\hat{\mathscr{F}}$ is the $\sigma$-field composed by all the subsets of $\hat \O$. Then, any map $\hat\O\to \mrn$ are $\hat{\mathscr{F}}$-measurable. Given $\{p_{i}\}_{i=1}^{m}$, $p_{i}>0$  for any $i=1,2,...,m$ and $\sum_{i=1}^{m}p_i=1$. Let $\hat P:\hat{\mathscr{F}}\to [0,1]$ be defined by $\hat P(\emptyset)=0$, $\hat P(\{\hat\o_i\})=p_i$, $i=1,2,...,m$ and
$$\hat P(A)=\sum_{\hat\o_i\in \hat A} p_i,\quad \forall \ A\in \hat{\mathscr{F}}.$$
 Then, $(\hat \O,\hat{\mathscr{F}}, \hat P)$ is a (discrete) probability space.

Let $\hat\xi:\hat\O\to \mrn$ be a random vector belonging to $\mathcal{L}^2(\hat \O,\hat{\mathscr{F}},\mrn)$. The $\hat\xi$ can be treated as a finite dimensional vector in the product Euclidean space $(R^{n})^m$ defined by
$\hat\xi=(\hat\xi_{1}^{\top},\hat\xi_{2}^{\top},...,\hat\xi_{m}^{\top})^{\top}$
with $\hat\xi_{i}=\hat\xi(\hat\o_i)$, $i=1,2,...,m$. The norm of $\xi$ is
\begin{eqnarray}
\|\hat\xi\|_{\mathcal{L}^2}&=&\left[\sum_{i=1}^{m}|\hat\xi_{i}|^2p_{i}\right]^{\frac{1}{2}} \nonumber\\
&=&\left[(\hat\xi_{1}^{\top},\hat\xi_{2}^{\top},...,\hat\xi_{m}^{\top})
\mathcal{P}(\hat\xi_{1}^{\top},\hat\xi_{2}^{\top},...,\hat\xi_{m}^{\top})^{\top}\right]^{\frac{1}{2}}, \end{eqnarray}
where
\begin{equation*}
\mathcal{P}:=
\begin{bmatrix}
p_1I&{}&{} \\
{}&\ddots&{}\\
{}&{}&p_{m}I
\end{bmatrix}
\end{equation*}
is a positive definite matrix in $R^{mn\times mn}$. Here $I\in R^{n\times n}$ is the $n\times n$-identity matrix.
In addition, for $\hat\xi,\hat\eta\in\mathcal{L}^2(\hat \O,\hat{\mathscr{F}},\mrn)$, the inner product of $\hat\xi$ and $\hat\eta$ is
$$\inner{\hat\xi}{\hat\eta}_{\mathcal{L}^2}=(\hat\xi_{1}^{\top},\hat\xi_{2}^{\top},...,\hat\xi_{m}^{\top})
\mathcal{P}(\hat\eta_{1}^{\top},\hat\eta_{2}^{\top},...,\hat\eta_{m}^{\top})^{\top}. $$
Therefore, $\mathcal{L}^2(\hat \O,\hat{\mathscr{F}},\mrn)$ is isomorphic to $(R^{n })^m$ (with norm $\|x\|_{\mathcal{P}}=x^{\top}\mathcal{P}x$ and inner product $ \inner{x}{y}_{\mathcal{P}}=x^{\top}\mathcal{P}y$ for any $x,y\in (R^{n })^m$). Similarly, in the discrete case, the space $\big[\mathcal{L}^2(\hat \O,\hat{\mathscr{F}},\mrn)\big]^3$ with norm $\|\cdot\|_{\mathcal{L}^2, G}$ and inner product$\inner{\cdot}{\cdot}_{\mathcal{L}^2, G}$ (See (\ref{normtheta}) and (\ref{innerproductthetav}) for their definitions) is also isomorphic  to a finite dimensional Hilbert space.

Since any bounded sequence in finite dimensional space is sequential compact,  when the sample space is a  finite set, the convergence of Algorithm 3.1 follows directly from Theorem \ref{theorem2} and \ref{theorem3}.

\begin{theorem}
Suppose the conditions in Theorem \ref{theorem2} hold true. Then, in the case that the sample space is a  finite set,
the sequence $\{\theta^k\}$ generated by Algorithm 3.1 converges to some $\theta^*\in Sol~\vi{T}{\mathcal{K}}$. Consequently, the  sequence $\{x^{k}\}_{k=0}^{\infty}$ converges to a solution $x^*$ to $\mbox{MSVI}(F,\mathcal{C}\cap\mathcal{N} )$ (\ref{SVIbasicdefinition}).
\end{theorem}

In the following we  give some test examples.

\begin{example}\label{example1}
Let $m\in \mathbb{N}$, $\hat\O:=\big\{\hat\o_1,\hat\o_2,...,\hat\o_m\big\}$. Generate randomly a discrete probability distribution $\hat P(\{\hat\o_i\})=p_i$, $i=1,2,...,m$ and numbers $\hat \xi_{i}\in R$, $i=1,2,...,m$. Let random variable  $\hat\xi:\hat \O\to R$ be defined by $\hat\xi(\hat\o_{i})=\hat \xi_{i}$, $i=1,2,...,m$ and   assume
$$\hat P(\hat\xi=\hat \xi_{i})=\hat P(\{\hat\o_{i}\})=p_{i}, i=1,2,...,m.$$
Let $\hat{\mathscr{F}}=2^{\hat\O}$, $(\hat \O,\hat{\mathscr{F}}, \hat P)$ and $\mathcal{L}^2(\hat \O,\hat{\mathscr{F}},\mrn)$ be defined as above.
Let $n_0,n_1\in \mathbb{N}$, $n_0+n_1=n$. For any $x\in \mathcal{L}^2(\hat \O,\hat{\mathscr{F}},\mrn)$, let $x(\hat \o_i)=(x_{0}(\hat \o_i), x_{1}(\hat \o_i))$, where $x_{0}(\hat \o_{i})\in R^{n_0}$, $x_{1}(\hat  \o_{i})\in R^{n_1}$, $i=1,2,...,m$. Let $\hat{\mathscr{F}}_{0}=\{\emptyset,\hat\O\}$, $\hat{\mathscr{F}}_{1}=\sigma(\hat{\xi})$. From the definition of $\hat\xi$, $\hat{\mathscr{F}}_{1}=2^{\hat\O}$ and
$$\mathcal{N}=\Big\{x:\hat\Omega\to R^n\ \Big|\ x(\hat\o_{i})=(x_{0}, x_1(\hat \xi(\hat\o_i))), i=1,2,...,m   \Big\},$$
i.e., the first $n_0$ components of $x\in \mathcal{N}$ is deterministic and equals to some $x_0\in R^{n_1}$, the last $n_1$ components of $x$ is a function of $\hat\xi$. In addition,
$$\mathcal{M}=\mathcal{N}^{\bot}=\Big\{y:\hat\Omega\to R^n\ \Big|\ y(\hat\o_{i})=(y_{0}(\hat\o_{i}),0), i=1,2,...,m,\dbE y_0= \sum_{i=1}^{m}y_{0}(\hat\o_{i})p_i=0  \Big\},$$
i.e., the first $n_0$ components of $y\in \mathcal{M}$ is a random vector taking values in $R^{n_0}$ with expectation 0, the last $n_1$ components of $y$ is identically  0.

Let $C(\hat\o_{i})\equiv[-1,1]^{n}$. Then,
$$\mathcal{C}=\big\{x:\hat\Omega\to R^n  \ \big|\ x(\hat\o_{i})\in [-1,1]^n,i=1,2,...,m\big\}.$$

Generate randomly nonzero positive semidefinite  matrices $M_{i}\in R^{n\times n}$ and vectors $b_{i}\in R^{n}$, $i=1,2,...,m$, and assume $M(\hat\xi_{i})=M_i$ and $b(\hat\xi_{i})=b_i$. Define $\hat F:\mathcal{L}^2(\hat \O,\hat{\mathscr{F}},\mrn)\to \mathcal{L}^2(\hat \O,\hat{\mathscr{F}},\mrn)$ by $$\hat F(x)=M(\hat\xi)x+b(\hat\xi),\quad \forall\ x\in \mathcal{L}^2(\hat \O,\hat{\mathscr{F}},\mrn).$$
It is easy to check that $\hat F$ is monotone and Lipschitz continuous.

Now we solve the $\mbox{MSVI}(\hat F,\mathcal{C}\cap\mathcal{N} )$   numerically by Algorithm 3.1. Choose arbitrarily
$$\theta^0=(x^0,y^0,\lambda^0)^{\top}\in \mathcal{C}\times\mathcal{N}\times\mathcal{L}^2(\hat \O,\hat{\mathscr{F}},\mrn).$$
Let $\{\theta^k\}_{k=0}^{\infty}=\{(x^k,y^k,\lambda^k)^{\top}\}_{k=0}^{\infty}$ be the sequence generated by Algorithm 3.1. By Lemma \ref{lemma 3.1}, $x^k$ is a solution to $\mbox{MSVI}(\hat F,\mathcal{C}\cap\mathcal{N} )$  if and only if $x^k=y^k$, $\lambda^k\in \mathcal{M}$ and
$$-\hat F(x^k)(\hat{\o}_i)+\lambda^k(\hat{\o}_i)\in N_{C(\hat\o_{i})}(x^k(\hat{\o}_i)),\quad   i=1,2,...,m.$$
That leads us to define the stopping criterion
\begin{align*}
Err(x^{k}):=\max_{i} |x^{k}(\hat\omega_{i})\!-\!\Pi_{C(\hat\o_{i})}(x^{k}(\hat\omega_{i})\!-\!\hat F(x^{k})(\hat\omega_{i})\!+\!\lambda^{k}(\hat\omega_{i}))|
\!+\!\sum_{i=1}^{m}| x^{k}(\hat\omega_{i})\!-\!y^{k}(\hat\omega_{i})|^2p_{i}<\varepsilon
\end{align*}
for some sufficiently small $\varepsilon>0$.

In the implementation of Algorithm 3.1, we choose $\alpha=0.61$, $\beta=1.1L_{\hat F}$ with
$$L_{\hat F}=\max \{\sigma_{i}\ |\ i=1,2,...,m \},$$
where $\sigma_{i}$ is the largest eigenvalue of $M_i$,$i=1,2,...,m$,
and, $r=1.1+(L_{\hat F}/\beta)$.

In Table \ref{tab test}, we report the numerical performance for PHA and Algorithm 3.1(Alg.3.1). We report the average number of iterations (Avg-iter) and the average running time in seconds (Avg-time(s)). It can be found from Table \ref{tab test} that the average running time of Algorithm 3.1 is shorter than that of PHA even though  the average number of iterations of Algorithm 3.1 is more than PHA.

\end{example}

\begin{table}[h]
\begin{center}
\begin{minipage}{\textwidth}
\caption{Numerical results for Example \ref{example1}}\label{tab test}
\begin{tabular*}{\textwidth}{@{\extracolsep{\fill}}lcccccc@{\extracolsep{\fill}}}
\toprule
$Err=10^{-3} $
&\multicolumn{2}{@{}c@{}}{$m=10, n_1=n_2=5 $}
&\multicolumn{2}{@{}c@{}}{$m=20, n_1=n_2=30$} \\
\cmidrule{2-3}\cmidrule{4-5}%
 & Avg-time(s) & Avg-iter  &    Avg-time(s) & Avg-iter  \\
\midrule
PHA  & 1.7628  & 28 & 35.0378 & 153\\
Alg. 3.1 & 0.2028  & 134  & 14.1337  & 2180 \\
\toprule
\end{tabular*}
\begin{tabular*}{\textwidth}{@{\extracolsep{\fill}}lcccccc@{\extracolsep{\fill}}}
$Err=10^{-5}$
& \multicolumn{2}{@{}c@{}}{$m=10, n_1=n_2=5 $ } & \multicolumn{2}{@{}c@{}}{$m=20, n_1=n_2=30$  } \\ \cmidrule{2-3}\cmidrule{4-5}%
 & Avg-time(s) & Avg-iter  &    Avg-time(s) & Avg-iter  \\
\midrule
PHA  & 34.0394  & 363 & 486.331 & 2001\\
Alg. 3.1 & 0.8424  & 322 & 24.6326  & 4086 \\
\toprule
\end{tabular*}
\end{minipage}
\end{center}
\end{table}

\begin{example}\label{example2}
Let $(\Omega, \mathscr{F},  \mathbb{F}, P)$  be a complete filtered probability space with the filtration $\mathbb{F}=\{\mathscr{F}_{t} \}_{0\le t\le 1}$, on which a one-dimensional standard  Wiener
process $W(\cdot)$ is defined such that $\mathbb{F}$ is the natural filtration generated by $W(\cdot)$ (augmented by all the P-null sets).

Let $n=m=1$,  $U=[0,1]$ and
$\eta\in \mathcal{L}^2(\Omega,\mathscr{F}_{1}, R)$.
Consider the discrete control system
\begin{equation}\label{eq dis_ex_W}
\left\{
\begin{aligned}
&x_{i+1}=x_{i}+[x_{i}-u_i]\Delta+ u_i\Delta W_i,\ i=0,1,...,N-1,\\[-0.2em]
&x_0=1
\end{aligned}
\right.
\end{equation}
and cost function
\begin{equation}\label{eq dis_cost}
J^N(\mathscr{U}^{N})=\frac{1}{2}\dbE|x_{N}-\eta|^2.
\end{equation}
Here,  $\mathscr{U}^{N}=(u_0,u_1,...,u_{N-1}):\O\to R^{N}$, $u_i(\o)\in U$, \mbox{a.s.} $\o\in\O$, $\mathscr{X}^{N}=(x_0,x_1,...,x_N):\O\to R^{N+1}$, $\Delta=1/N$, $\Delta W_{i}=W((i+1)/N)-W(i/N)$, $i=0,1,...,N-1$ are independent identically distributed Gaussian random variables with mean $0$ and variance $1/N$.

Clearly, the discrete-times stochastic optimal control problem \eqref{disc SOCP} in Example \ref{ex SOCP} with control system \eqref{eq dis_ex_W} and cost function \eqref{eq dis_cost} is a discretized approximation of the continuous-time stochastic optimal control problem \eqref{SOCP} with control system
$$
\left\{
\begin{aligned}
&dx(t)=(x(t)-u(t))dt+ u(t)dW(t),\ t\in[0,1],\\[-0.2em]
&x_0=1
\end{aligned}
\right.
$$
and cost function
$$
J(u)=\frac{1}{2}\dbE|x(1)-\eta|^2.
$$

It follows from Example \ref{ex SOCP} that, the first order necessary condition for the discrete-times optimal control problem \eqref{disc SOCP} is a multistage stochastic variational inequality defined on a general probability space. In order to obtain an approximation of that multistage stochastic variational inequality in a discrete sample space, we use a sequence of random walks to approximate the Wiener process.

Let us consider a sequence $\{\xi_{i}\}_{i=1}^{N\ell}$ ($\ell\in \mathbb{N}$) of independent  identically distributed random variables such that for each $i$, $P(\xi_{i}=-1)=P(\xi_{i}=1)=0.5$, and, we define
$$S_{0}=0,\  S_{i}=\sum_{k=1}^{i}\xi_k,\quad i=1,2,...,N\ell$$
and
$$Y_{0}=0,\  Y_{i}^{\ell}=\frac{1}{\sqrt{N\ell}}S_{i\ell},\quad i=1,2,...,N.$$
By \cite[Theorem 4.17, Chapter 2]{Karatzas1991}, $(Y_{1}^{\ell},Y_{2}^{\ell},...,Y_{N}^{\ell})$ converges to $(W(1/N),W(2/N),...,W(1))$ in law as $\ell\to \infty$. Define
\begin{equation}\label{sample space ex2}
\hat \O=  \underbrace{\{-1,1\}\times\{-1,1\}\times\cdots\times\{-1,1\}}_{N\ell},
\end{equation}
and let $\hat{\mathscr{F}}$ be all the subsets of $\hat\O$ and $\hat P$ be the probability measure induced by the binomial distribution. Then, $(\hat \O, \hat{\mathscr{F}},  \hat P)$ is a discrete probability space.  Denote $\Delta Y_{i}^{\ell}:=(S_{(i+1)\ell}-S_{i\ell})/\sqrt{N\ell}$, $i=0,1,...,N-1$ and let $\hat{\mathscr{F}}_0=\{\emptyset, \O\}$, $\hat{\mathscr{F}}_{i}=\sigma(\Delta Y_0^{\ell},\Delta Y_1^{\ell},...,\Delta Y_{i-1}^{\ell} ) ,\   i=1,2,...,N-1$. We define
$$
\mathcal{N}:=\Big\{\hat {\mathscr{U}}^{N}=(\hat u_{0},\hat u_{1},...,\hat u_{N-1})\in \mathcal{L}^2(\hat\Omega,\hat{\mathscr{F}}, R^{N}) \ \big|\ \hat u_{i}\in \mathcal{L}^0(\hat\Omega, \hat{\mathscr{F}}_{i}, R), i=0,1,...,N-1\Big\}
$$
and
\begin{equation*}
\mathcal{C}:=\Big\{\hat {\mathscr{U}}^{N}=(\hat u_{0}, \hat u_{1},...,\hat u_{N-1})\in \mathcal{L}^2(\Omega,\hat{\mathscr{F}}, R^{N})\ \Big|\ \hat u_{i}(\omega)\in U \ \mbox{a.s. } \o\in\O,\ i= 0,1,...,N-1 \Big\}.
\end{equation*}

Replacing $\Delta W_{i}$ by $\Delta Y_{i}^{\ell}$, $i=0,1,...,N-1$ in \eqref{eq dis_ex_W}, we obtain the following stochastic approximation difference equation
\begin{equation}\label{eq dis_ex_randwalk}
\left\{
\begin{aligned}
&\hat{x}_{i+1}=\hat{x}_{i}+[\hat{x}_{i}-\hat{u}_i]\Delta+ \hat{u}_i\Delta Y_{i}^{\ell}  ,\ i=0,1,...,N-1,\\[-0.2em]
&\hat{x}_0=1.
\end{aligned}
\right.
\end{equation}
Similarly to Example \ref{ex SOCP}, we  denote $\hat\Psi_{i}=1+\Delta$, $\hat\Lambda_{i}=-\Delta+\Delta Y_i^{\ell}$, $i=0,1,...,N-1$ and define
$$ \hat Z_{i}=\Big[\prod_{j=i+1}^{N-1}\hat{\Psi}_{j}\Big]\hat{\Lambda}_{i},\ i=0,1,..., N-1, \ \hat\zeta=\sum_{i=0}^{N-1}\hat Z_{i}.$$
Then, $\hat \eta=\prod_{i=0}^{N-1}\hat{\Psi}_{i}+ \hat \zeta$ is the final value of the solution to \eqref{eq dis_ex_randwalk} with control $\hat {\mathscr{U}}^{N}\equiv(1,1,...,1)$.

Define the random walk approximation of cost function \eqref{eq dis_cost} by
\begin{equation}\label{eq disc cot randwalk}
\hat J^N(\hat {\mathscr{U}}^{N})
\!=\!\frac{1}{2}\dbE\big|\hat x_{N}- \hat\eta \big|^2=\frac{1}{2}\dbE\big|\sum_{i=0}^{N-1}\hat Z_{i}\hat{u}_{i}- \hat\zeta \big|^2.
\end{equation}
The discretized approximation stochastic optimal control problem of random walks is: To find $ (\hat {\mathscr{U}}^{N})^*\in \mathcal{C}\cap \mathcal{N}$ such that
\begin{equation}\label{disc SOCP randomwalk}
\hat J^N((\hat {\mathscr{U}}^{N})^*)=\min_{\hat {\mathscr{U}}^{N}\in \mathcal{C}\cap \mathcal{N}}\hat J^N(\hat {\mathscr{U}}^{N}).
\end{equation}
Clearly, $(\hat {\mathscr{U}}^{N})^*\equiv(1,1,...,1)$ is the optimal solution to \eqref{disc SOCP randomwalk}.

Letting
\begin{equation}\label{eq dis_Matrix}
M=(\hat Z_0,\hat Z_1,..., \hat Z_{N-1})^{\top}(\hat Z_0,\hat Z_1,..., \hat Z_{N-1}),\ b=\hat\zeta(\hat Z_0,\hat Z_1,..., \hat Z_{N-1})^{\top},
\end{equation}
we have
$$D \hat J^N(\hat {\mathscr{U}}^{N})=M\hat{\mathscr{U}}^{N}-b, \quad\ \forall\ \hat {\mathscr{U}}^{N}\in \mathcal{L}^2(\hat\Omega,\hat{\mathscr{F}}, R^{N}) .$$
To solve \eqref{disc SOCP randomwalk} we only need to solve the multistage stochastic variational inequality
\begin{equation}\label{SVI_SOCP EX}
D \hat J^N((\hat {\mathscr{U}}^{N})^*)\in N_{\mathcal{C}\cap\mathcal{N}}((\hat {\mathscr{U}}^{N})^*).
\end{equation}

Note that the sample space $\hat \O$ defined by \eqref{sample space ex2} has $2^{N\ell}$ sample points. In order to obtain a relatively high approximation accuracy for the original discrete-time optimal control problem (or related continuous-time optimal control problem), the number of the independent identically distributed random variables  $\ell$ for approximating the increment of Wiener process (and the number of the partitions $N$) should be large enough. Then, $\hat \O$  will contain  an extremely large number of sample points and there is no hope to obtain the exact solution  $\tilde y$ of the projection onto the  nonanticipativity subspace $\mathcal{N}$ (in Step 1 of Algorithm 3.1). The Monte Carlo method is used to calculate  $\tilde y$. Consequently, the calculation for $\tilde y$ will be time-consuming and it is difficult to reduce the calculation error. Let $\kappa(\in \mathbb{N})$ be the number of  sample size. The numerical results of \eqref{SVI_SOCP EX} is reported in Table 2. The parameters are determined in the same way as that in Example \ref{example1}.
\end{example}

\begin{table}[h]
\begin{center}
\begin{minipage}{\textwidth}
\caption{Numerical results for Example \ref{example2}}\label{tab test 2}
\begin{tabular*}{\textwidth}{@{\extracolsep{\fill}}lcccccc@{\extracolsep{\fill}}}
\toprule
$Err=10^{-3}$
&\multicolumn{2}{@{}c@{}}{$N=10,\ell=50,  \kappa=1000$}
&\multicolumn{2}{@{}c@{}}{$N=20, \ell=100, \kappa=2000$} \\
\cmidrule{2-3}\cmidrule{4-5}%
 & Avg-time(s) & Avg-iter  &    Avg-time(s) & Avg-iter  \\
\midrule
PHA  & 512.1321  & 12876 & 2467.1321 & 23124\\
Alg. 3.1 & 301.6875  & 20127  & 1190.5237  & 34566 \\
\toprule
\end{tabular*}
\end{minipage}
\end{center}
\end{table}

\section{Concluding remark}

This paper is devoted to establishing an explicit type splitting algorithm for multistage stochastic variational inequalities based on the prediction-correction ADMM for deterministic variational inequalities with separable structures. As we have seen, the main difference between the deterministic variational inequality and the  multistage stochastic variational inequality is that, in the stochastic case, there exists an extra nonanticipativity constraint which leads to some new difficulty in proposing a proper algorithm for multistage stochastic variational inequalities. The key idea of both PHA and Algorithm 3.1 is to treat the projections onto the nonempty closed convex set and the nonanticipativity subspace separately by proper splitting method. The main advantage of Algorithm 3.1 is that, it is an explicit iterative algorithm so that the calculation in each step of the algorithm becomes much easier.

In order to simplify the discussion and make the main idea much clear,  we simplified the original prediction-correction ADMM (for deterministic variational inequalities) in \cite{HeJCM2006}. It should be remarked that, the algorithm proposed in this paper can be further developed and generalized. For instance, the prediction-correction ADMM with variable parameters $\beta$ and/or $\alpha$, some accelerated algorithms based on the prediction-correction ADMM. Furthermore, by Lemma \ref{lemma 3.1}, under proper conditions, solving  original multistage stochastic variational inequality $\mbox{MSVI}(F,\mathcal{C}\cap\mathcal{N} )$ is equivalent to solving the variational inequality \vi{T}{\mathcal{K}}. That give us an opportunity to solve multistage stochastic variational inequalities with some other algorithms for deterministic variational inequalities in infinite dimensional spaces. Some new algorithms  for multistage stochastic variational inequalities in the general probability space might be proposed in that way. Further more,  some    algorithms with strong convergence for multistage stochastic variational inequalities in the general probability space are also valuable for further research.

The multistage stochastic variational inequality in this paper is defined on $\mathcal{L}^2(\Omega,\mathscr{F}, R^n)$. One of the main motivation to define  the multistage stochastic variational inequality on that space is for its applications in stochastic optimal control problems, in which the admissible controls are usually chosen to be the square-integrable stochastic processes (see Example \ref{ex SOCP}). In addition, since   $\mathcal{L}^2(\Omega,\mathscr{F}, R^n)$ is a Hilbert space, some technical difficulties in the algorithm design and the convergence analysis are avoided when we consider the  multistage stochastic variational inequality on that space. It should be remarked that, in some original research articles, the two-stage stochastic programming problems are considered  on $\mathcal{L}^\infty(\Omega,\mathscr{F}, R^n)$, see for instance \cite{Rockafellar wets 1976} or the book \cite{Shapiro handbook SP 2009}. Therefore, it is interesting to investigate the multistage stochastic variational inequality on $\mathcal{L}^\infty(\Omega,\mathscr{F}, R^n)$. Also, we can study the  multistage stochastic variational inequality on $\mathcal{L}^p(\Omega,\mathscr{F}, R^n)$ for any $p\in[1,\infty]$. When $p\neq 2$, the space $\mathcal{L}^p(\Omega,\mathscr{F}, R^n)$ is not a Hilbert space (even not a reflexive Banach space if $p=1$ or $p=\infty$), some new phenomena and new difficulties might appear. It shall be investigated elsewhere.

\vspace{+2em}

\textbf{Acknowledgements.} Both authors would like to thank the referees and the Associate Editor for their critical comments and helpful suggestions.

\vspace{+1em}

\appendix

\section{Proof of Lemma \ref{appendix6}}

In the appendix, we shall  give the proof of Lemma \ref{appendix6}.  The main idea of the proof comes from \cite[Lemma 4.6]{wangzhang17}.

\begin{proof}
It is clear that  \eqref{VI pointwise} implies \eqref{VI integral}. Then we only need to prove that the converse is also true.

Define
\begin{align*}
 A=\{\omega\in\Omega\mid\exists~x\in C(\omega)~s.t.~\langle F(x^*)(\omega), x-x^*(\omega)\rangle<0\}.
\end{align*}
To prove that ~\eqref{VI pointwise} holds, we only need to prove that ~$P(A)=0$.
Let
$$
 G=\{(\omega,x)\in\Omega\times R^n \mid x\in C(\omega),~\langle F(x^*)(\omega), x-x^*(\omega)\rangle<0\}.
$$
Then $G$ is $\mathscr{F}\otimes\mathscr{B}(R^n)$-measurable.By \cite[Theorem III.23]{Castaing1977}, $A$ is $\mathscr{F}$-measurable.

Take $k,r=1,2,\cdots$, define
$$
 A_{k,r}=\{\o\in\O \mid \exists~x\in C(\omega)\cap\bar{B}(0,r) ~s.t.~\langle F(x^*)(\omega), x-x^*(\omega)\rangle\le-\frac{1}{k} \}
$$
and
$$
  \Phi_{k,r}(\omega)=\{x\in C(\omega)\cap\bar{B}(0,r)\mid \langle F(x^*)(\omega),x-x^*(\omega)\rangle\le-\frac{1}{k}\}.
$$
Here $\bar{B}(0,r)$ is the closed ball in $R^n$ of center $0$ and radius $r$.
Similarly,  $A_{k,r}$ is $\mathscr{F}$-measurable. In addition, $\Phi_{k,r}$ is an $\mathscr{F}$-measurable set valued map  and
$$
A=\bigcup_{r=1}^{\infty}\bigcup_{k=1}^{\infty}A_{k,r}.
$$
To prove $P(A)=0$, we only need to prove that for any $k,r$, $P(A_{k,r})=0$. Assume that there exist $k,r$ such that $P(A_{k,r})>0$.
By Lemma \ref{measurable-selec}, there exists $\eta\in\mathcal{L}^2(\Omega,\mathscr{F},P)$ such that
\begin{align*}
\eta(\omega)\in\Phi_{k,r}(\omega),\quad \mbox{a.s.}~\omega\in A_{k,r}.
\end{align*}
Define $\tilde{\eta}=\eta\chi_{A_{k,r}}+x^* \chi_{\O\setminus A_{k,r}}$, then $\tilde{\eta}\in\mathcal{C}$, and
\begin{align*}
\langle F(x^*),\tilde{\eta}-x^*\rangle_{\mathcal{L}^2}
=&\int_{\Omega}\langle F(x^*)(\omega),\tilde{\eta}(\omega)-x^*(\omega)\rangle P(d(\omega))  \nonumber\\
=&\int_{A_{k,r}}\langle F(x^*)(\omega),\eta(\omega)-x^*(\omega)\rangle P(d(\omega))\nonumber\\
\le&-\frac{1}{k}P(A_{k,r})\nonumber\\
<&0,
 \end{align*}
contradicting  \eqref{VI integral}. Thus $P(A)=0$.
\end{proof}

\end{spacing}

\end{document}